\documentclass[submission,copyright,creativecommons]{eptcs}
\providecommand{\event}{QPL 2016}
\providecommand{\titlerunning}{Paschke Dilations}
\providecommand{\authorrunning}{Westerbaan \& Westerbaan}
\usepackage{breakurl}

\makeatletter
\newcommand{\killpic}{%
  \hangindent=0pt
  \let\par=\old@par
}
\makeatother

\usepackage{xypic}
\usepackage{xspace}
\usepackage{amsthm}
\usepackage{picins}
\usepackage{amsmath}
\usepackage{amssymb}
\usepackage{relsize}
\usepackage{paralist}
\usepackage{nicefrac}
\usepackage{multicol}
\usepackage[warn]{textcomp}
\usepackage{mathrsfs}
\usepackage{hyperref}
\usepackage{mathtools}
\usepackage[obeyDraft,
            color=gray,
            backgroundcolor=white,
            textsize=footnotesize,
            bordercolor=gray]{todonotes}

\newcommand{\C}{\mathbb{C}}
\newcommand{\Ad}{\operatorname{\mathrm{Ad}}}
\newcommand{\after}{\mathbin{\circ}}
\newcommand{\keyword}[1]{\textbf{#1}}
\newcommand{\id}{\mathrm{id}}
\newcommand{\NCP}{\mathrm{NCP}}
\newcommand{\ft}{\measuredangle}

\newcommand{\A}{B^a}
\newcommand{\car}{\mathop{\mathrm{car}}}

\DeclarePairedDelimiter\floor{\lfloor}{\rfloor}
\DeclarePairedDelimiter\ceil{\lceil}{\rceil}

\newcounter{main}
\newtheorem{thm}[main]{Theorem}
\newtheorem{lem}[main]{Lemma}
\newtheorem{prop}[main]{Proposition}
\newtheorem{cor}[main]{Corollary}

\newtheorem{prob}[main]{Problem}

\theoremstyle{definition}
\newtheorem{dfn}[main]{Definition}
\newtheorem{exa}[main]{Example}

\newtheorem{rem}[main]{Remark}
\newtheorem{overview}[main]{Overview}

\title{Paschke Dilations}

\author{Abraham Westerbaan
        \institute{Radboud Universiteit Nijmegen}
        \email{bram@westerbaan.name}
        \and Bas Westerbaan
        \institute{Radboud Universiteit Nijmegen}
        \email{bas@westerbaan.name}}

\begin{document}
    
\maketitle

\begin{abstract}
In 1973
Paschke defined a factorization for
completely positive maps
between~C$^*$-algebras.
In this paper
we show that
for normal maps between von Neumann algebras,
this factorization has a universal property,
and coincides with Stinespring's dilation
for normal maps into~$B(\mathscr{H})$.
\end{abstract}
The Stinespring Dilation Theorem\cite{stinespring}  entails
that every normal completely positive linear map (\emph{NCP-map})
$\varphi\colon \mathscr{A} \to B(\mathscr{H})$
is of the form
$\xymatrix{\mathscr{A} \ar[r]|-\pi
& B(\mathscr{K}) \ar[rr]|-{ V^*(\,\cdot\,) V}
&& B(\mathscr{H}) }$
where~$V\colon\mathscr{H}\to \mathscr{K}$ is a bounded operator
and~$\pi$ a normal unital~$*$-homomorphism (\emph{NMIU-map}).
Stinespring's theorem
is fundamental in the study
of quantum information and quantum computing:
it is used to prove entropy inequalities (e.g.~\cite{lindblad}),
bounds on optimal cloners (e.g.~\cite{werner}),
full completeness of quantum programming languages (e.g.~\cite{staton}),
security of quantum key distribution (e.g.~\cite{werner2}),
analyze quantum alternation (e.g.~\cite{prakash}),
to categorify quantum processes (e.g.~\cite{selinger}) \emph{and}
as an axiom to single out
quantum theory among information processing theories.\cite{chiribella}
A fair overview of all uses of Stinespring's theorem and its consequences
would warrant a separate article of its own.

One wonders:
is the Stinespring dilation categorical in some way?
Can the Stinespring dilation theorem be generalized to arbitrary
NCP-maps~$\varphi\colon \mathscr{A} \to \mathscr{B}$?
In this paper we answer both questions in the affirmative.
We use the dilation
introduced by Paschke\cite{bew154}
for arbitrary NCP-maps, and
we show that it coincides
with Stinespring's dilation
(a fact not shown before)
by introducing a universal property for Paschke's dilation,
which Stinespring's dilation also satisfies.

In the second part of this paper,
we will study the class of maps that may appear on the right-hand side
of a Paschke dilation,
to prove the counter-intuitive
fact that both maps in a Paschke dilation
are extreme (among NCP maps with same value on~$1$).

Let us give the universal property
and examples right off the bat;
proofs are further down.
\pichskip{1em}
\parpic[r]{
$\xymatrix@C=2.5em@R=2em{
\mathscr{A} 
\ar[rr]^\varphi
\ar[rd]_{\varrho}
\ar@/_1.5em/[rdd]_{\varrho'}
&
& 
\mathscr{B} 
\\
&
\mathscr{P}
\ar[ru]_{f}
&
\\
&
\mathscr{P}'
\ar@/_1.5em/[ruu]_{f'}
\ar@{-->}[u]_-{\sigma}
}$}
\pichskip{2em}
\begin{thm}
Every NCP-map $\varphi\colon\mathscr{A}\to\mathscr{B}$
has a \keyword{Paschke dilation}.
A Paschke dilation of~$\varphi$
is a pair of maps
$\xymatrix{\mathscr{A}\ar[r]|\varrho & \mathscr{P}
\ar[r]|f & \mathscr{B}}$,
where $\mathscr{P}$ is a von Neumann algebra,
$\varrho$ is an NMIU-map
and~$f$ is an NCP-map with~$\varphi = f \after \varrho$
such that for every
other~$\xymatrix{\mathscr{A}\ar[r]|{\varrho'} & \mathscr{P}'
\ar[r]|{f'} & \mathscr{B}}$,
where $\mathscr{P}'$ is a von Neumann algebra,
$\varrho'$ is an NMIU-map,
and~$f'$ is an NCP-map with~$\varphi=f'\after\varrho'$,
there is a unique NCP-map 
$\sigma\colon \mathscr{P}'\to\mathscr{P}$
such that the diagram on the right commutes.
\end{thm}
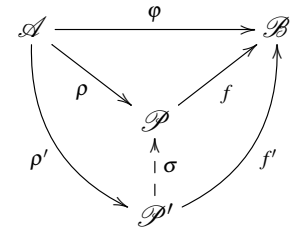
\killpic
\begin{exa}
A \emph{minimal} Stinespring dilation
$\xymatrix{
\mathscr{A}
\ar[r]|-\pi
&
B(\mathscr{K})
\ar[rr]|-{ V(\,\cdot\,) V^*}
&&
B(\mathscr{H})
}$ of an NCP-map is a Paschke dilation,
see Theorem~\ref{thm:stinespring-paschke}.
\end{exa}
\begin{exa}
As a special case of the previous example,
we see the GNS construction
for a normal state~$\varphi$
on a von Neumann algebra~$\mathscr{A}$,
gives a Paschke dilation
$\xymatrix{
\mathscr{A}
\ar[r]|-\pi
&
B(\mathscr{H})
\ar[rr]|-{\left<\xi,(\,\cdot\,)\xi\right>}
&&
\mathbb{C}
}$ of~$\varphi$.
In particular,
the Paschke dilation of $(\lambda,\mu)=\frac{1}{2}(\lambda+\mu),
\ \mathbb{C}^2\to\mathbb{C}$
is 
\begin{equation*}
\xymatrix@C=10em{
\mathbb{C}^2\ar[r]^-{(\lambda,\mu)\mapsto
\left(\begin{smallmatrix}\lambda & 0 \\ 0 & \mu \end{smallmatrix}
\right)}  
&  M_2 \ar[r]^-{%
\left(\begin{smallmatrix}a & b \\ c & d\end{smallmatrix}%
\right)\mapsto \frac{1}{2}(a+b+c+d)} & \mathbb{C} }.
\end{equation*}
This gives
a universal property to
the von Neumann algebra~$M_2$ of $2\times 2$ complex matrices,
(which is a  model of the qubit.)
\end{exa}
The following examples can be proven
using only the universal property of a Paschke dilation.
\begin{exa}
The Paschke dilation of an NMIU-map~$\varrho\colon \mathscr{A}\to\mathscr{B}$
is
$\xymatrix{
\mathscr{A}
\ar[r]|-\varrho
&
\mathscr{B}
\ar[r]|-{\mathrm{id}}
&
\mathscr{B}
}$.
\end{exa}
\begin{exa}\label{exa:pure}
If
$\xymatrix{\mathscr{A}
\ar[r]|-{\varrho} & \mathscr{P}
\ar[r]|-{f} & \mathscr{B}}$ is a Paschke dilation,
then
$\xymatrix{\mathscr{P}
\ar[r]|-{\id} & \mathscr{P}
\ar[r]|-{f} & \mathscr{B}}$ is a Paschke dilation of~$f$.
\end{exa}

\begin{exa}
Let~$\varphi\colon \mathscr{A} \to \mathscr{B}_1 \oplus \mathscr{B}_2$
be any NCP-map.
$\xymatrix@C+1pc{\mathscr{A} \ar[r]|-{\left<\varrho_1,\varrho_2\right>} &
\mathscr{P}_1 \oplus \mathscr{P}_2
\ar[r]|-{f_1 \oplus f_2} & \mathscr{B}_1 \oplus \mathscr{B}_2}$
is a Paschke dilation of~$\varphi$
if
$\xymatrix{\mathscr{A} \ar[r]|-{\varrho_i} &
\mathscr{P}_i
\ar[r]|-{f_i} & \mathscr{B}_i}$
is a Paschke dilation of~$\pi_i \after \varphi$ for~$i =1,2$.
\end{exa}
\begin{exa}
Let~$\varphi\colon \mathscr{A} \to \mathscr{B}$ be any~NCP-map
with Paschke dilation
$\xymatrix{\mathscr{A} \ar[r]|-{\varrho} &
\mathscr{P}
\ar[r]|-{f} & \mathscr{B}}$ and~$\lambda >0$.
Then
$\xymatrix{\mathscr{A} \ar[r]|-{\varrho} &
\mathscr{P}
\ar[r]|-{\lambda f} & \mathscr{B}}$
is a Paschke dilation of~$\lambda \varphi$.
\end{exa}
\begin{exa}
Let
$\xymatrix{\mathscr{A} \ar[r]|-{\varrho} &
\mathscr{P}
\ar[r]|-{f} & \mathscr{B}}$ be a Paschke dilation for a
map~$\varphi\colon\mathscr{A} \to \mathscr{B}$.
If~$\vartheta\colon \mathscr{P} \to \mathscr{P}'$ is any isomorhism,
then~$\xymatrix@C+.9pc{\mathscr{A} \ar[r]|-{\vartheta \after \varrho} &
\mathscr{P}'
\ar[r]|-{f\after \vartheta^{-1}} & \mathscr{B}}$ is also a Paschke dilation
of~$\varphi$.
\end{exa}
There is a converse to the last example:
\begin{lem}\label{lem:dilationsareisomorphic}
If
$\xymatrix{\mathscr{A}
\ar[r]|-{\varrho_i} & \mathscr{P}_i
\ar[r]|-{f_i} & \mathscr{B}}$ ($i=1,2$)
are Paschke dilations for the same
map~$\varphi\colon \mathscr{A} \to \mathscr{B}$,
then there is a unique (NMIU)
isomorphism~$\vartheta\colon \mathscr{P}_1 \to \mathscr{P}_2$
such that~$\vartheta \after \varrho_1 = \varrho_2$
and~$f_2 \after \vartheta = f_1$.
\end{lem}
\begin{proof}
There are unique mediating maps~$\sigma_1\colon \mathscr{P}_1 \to \mathscr{P}_2$
and~$\sigma_2 \colon \mathscr{P}_2 \to \mathscr{P}_1$.
It is easy to see~$\sigma_1 \after \sigma_2$
satisfies the same property as the unique mediating
map~$\id\colon \mathscr{P}_1 \to \mathscr{P}_1$ and
so~$\sigma_1 \after \sigma_2 = \id$.
Similarly~$\sigma_2 \after \sigma_1 = \id$.
Define~$\vartheta = \sigma_1$.
We just saw~$\vartheta$ is an NCP-isomorphism.
Note~$\vartheta(1) = \vartheta(\varrho_1(1)) = \varrho_2(1) = 1$
and so~$\vartheta$ is unital.
But then by~\cite[Corollary 47]{ww16}
$\vartheta$ is an NMIU isomorphism.
\end{proof}

\section{Two universal properties for Stinespring's dilation}
Let~$\varphi\colon \mathscr{A} \to B(\mathscr{H})$
be a NCP-map where~$\mathscr{A}$
is a von Neumann algebra
and~$\mathscr{H}$ is a Hilbert space.
In this section,
we prove that any minimal normal Stinespring dilation
of~$\varphi$
gives a Paschke dilation of~$\varphi$.
Let us first recall the relevant definitions.

\begin{dfn}
A \keyword{normal Stinespring dilation} of~$\varphi$,
is a triple~$(\mathscr{K}, \pi, V)$,
where~$\mathscr{K}$ is a Hilbert space,
$\pi\colon \mathscr{A} \to B(\mathscr{K})$ is an NMIU-map,
and~$V\colon \mathscr{H} \to \mathscr{K}$ a bounded operator
such that~$\varphi = \Ad_V \after \pi$,
where $\Ad_v\colon B(\mathscr{K})\to B(\mathscr{H})$
is the NCP-map given by~$\Ad_V (A) = V^* A V$
for all~$A\in B(\mathscr{K})$.\footnote{Be warned:
many authors
prefer to define~$\Ad_V$ by~$\Ad_V(A)= VAV^*$ instead.}
If the linear span of
$\{\,\pi (a) V x\colon \,a\in\mathscr{A}\,x\in\mathscr{H}\,\}$
is dense in~$\mathscr{K}$,
then  $(\mathscr{K},\pi,V)$
is called \keyword{minimal}.
\end{dfn}

It is a well-known fact that
all minimal normal Stinespring dilations of~$\varphi$
are unitarily equivalent (see e.g.~\cite[Prop.~4.2]{paulsen}).
We will adapt its proof
to show that a minimal Stinespring dilation
admits a universal property (Prop.~\ref{prop:stinespring-spatial}),
which we will need later on.
The adaptation is mostly straight-forward,
except for the following lemma.
\parpic[r]{
$\xymatrix{\mathscr{A} \ar[r]^\pi \ar[rd]_{\pi'}&
        \mathscr{B} \\
        & \mathscr{C} \ar[u]_\sigma}$}
\begin{lem}\label{lem:choi}
    Let~$\pi \colon \mathscr{A} \to \mathscr{B}$,
    $\pi' \colon \mathscr{A} \to \mathscr{C}$
        be NMIU-maps between von Neumann algebras,
	and let  $\sigma\colon \mathscr{C} \to \mathscr{B}$
	be an NCP-map
        such that~$\sigma \after \pi' = \pi$. \\
    Then $\sigma(\,\pi'(a_1) \,c\, \pi'(a_2)\,)
            = \pi(a_1) \,\sigma(c)\, \pi(a_2)$
        for any~$a_1,a_2 \in \mathscr{A}$ and~$c \in \mathscr{C}$.
\end{lem}
\killpic
\begin{proof}
By Theorem~3.1 of~\cite{aaw64},
we know that, for all~$c,d\in\mathscr{C}$,
\begin{equation}
\label{eq:choi}
\sigma(d^*d)\,=\,\sigma(d)^*\sigma(d)\quad\implies\quad
\sigma(cd)\,=\,\sigma(c)\sigma(d).
\end{equation}
Let~$a \in \mathscr{A}$.
We have~$\sigma(\pi'(a)^* \pi'(a))
    = \sigma(\pi'(a^* a))
    = \pi(a^* a)
    = \pi(a)^* \pi(a)
    = \sigma(\pi'(a))^* \sigma(\pi'(a))$.
By~\eqref{eq:choi},
we have~$\sigma(c\pi'(a)) = \sigma(c)\sigma(\pi'(a)) \equiv \sigma(c)\pi(a)$
for all~$c \in \mathscr{C}$.
Then also $\sigma(\pi'(a) c) = \pi(a)\sigma(c)$
for all~$c\in \mathscr{C}$ 
	(by taking adjoints).
Thus $\sigma(\,\pi'(a_1)\,c\,\pi'(a_2)\,)
= \pi(a_1)\,\sigma(c\pi'(a_2))
= \pi(a_1)\,\sigma(c)\,\pi(a_2)$
for all~$a_1,a_2\in \mathscr{A}$
and~$c\in\mathscr{C}$.
\end{proof}
\begin{lem}
\label{lem:adv}
Let~$\mathscr{K}$
be a Hilbert space.
If~$\Ad_S = \Ad_T$
for~$S,T\in B(\mathscr{K})$,
then $S=\lambda T$
for some $\lambda\in\mathbb{C}$.
\end{lem}
\begin{proof}
Let~$x\in \mathscr{K}$ be given,
and let~$P$ be the projection 
onto~$\{\lambda x\colon \lambda\in\mathbb{C}\}$.
Then
\begin{equation*}
\{\,\lambda S^*x\colon\,\lambda\in\mathbb{C}\,\}
\ =\ 
\mathrm{Ran}(\ S^*PS \ )
\ = \ 
\mathrm{Ran}(\ T^*PT \ )
\ = \ 
\{\,\lambda T^*x\colon\,\lambda\in\mathbb{C}\,\}.
\end{equation*}
It follows that~$S^*x = \alpha T^* x$
for some~$\alpha\in \mathbb{C}$
	with~$\alpha\neq 0$.
While $\alpha$ might depend on~$x$,
there is
$\alpha_0\in\mathbb{C}$
with~$\alpha_0\neq 0$ and~$S^* = \alpha_0 T^*$
by Lemma~9 of~\cite{ww16}.
Then~$S=\alpha_0^* T$.
\end{proof}

\begin{prop}\label{prop:stinespring-spatial}
    Let~$(\mathscr{K},\pi,V)$ and
        $(\mathscr{K}', \pi', V')$
	be normal Stinespring dilations of~$\varphi$.
	If~$(\mathscr{K},\pi,V)$
	is minimal,
	then
        there is a unique isometry~$S \colon \mathscr{K} \to \mathscr{K}'$
        such that~$S V =  V'$ and~$\pi = \Ad_S \after \pi'$.
\end{prop}
\begin{proof}
Let us deal with a pathological case.
    If~$V=0$,
        then~$\varphi=0$, $V'=0$, $\mathscr{K}=\{0\}$ and~$\pi = 0$,
	and so the unique linear map~$S \colon \{0\} \to \mathscr{K}'$
        satisfies the requirements.  Assume~$V \neq 0$.
        
\noindent \emph{(Uniqueness)}\ 
    Let~$S_1,S_2\colon \mathscr{K} \to \mathscr{K}'$
    be isometries
    with~$S_k V = V'$ and~$\Ad_{S_k} \after \pi' = \pi$.
    We must show that~$S_1=S_2$.
    We have, for all,
	    $a_1,\dotsc,a_n,\alpha_1,\dotsc,\alpha_n\in \mathscr{A}$,
	    $x_1,\dotsc,x_n,y_1,\dotsc,y_n\in \mathscr{H}$,
    and~$c\in \mathscr{C}$,
\begin{align*}
    \bigl< \ \Ad_{S_k} (c) \sum_i \pi(a_i) V x_i,\ 
                            \sum_j  \pi(\alpha_j) V y_j \ \bigr> 
        & \ =\ \sum_{i,j} \bigl< \ V^* \pi(\alpha_j^*)\,\Ad_{S_k} (c)\, \pi(a_i) V x_i,\ 
            y_j \ \bigr>  && \text{by rearranging}\\
        & \ =\  \sum_{i,j }\bigl< \ V^*  \Ad_{S_k }(\ \pi'(\alpha_j^*) 
		\,c\, \pi'(a_i)\ ) V x_i,\ 
            y_j \ \bigr>  &\quad& \text{by Lemma~\ref{lem:choi}}\\
        & \ =\  \sum_{i,j }\bigl< \ (V')^*   \pi'(\alpha_j^*) \,c\, \pi'(a_i) V' x_i,\ 
            y_j \ \bigr>  &\quad& \text{as $S_kV=V'$}
\end{align*}
Since the linear span 
of~$\pi(\mathscr{A})V\mathscr{H}$ is dense in~$\mathscr{K}$,
we get~$\Ad_{S_1} = \Ad_{S_2}$.
    Thus~$\lambda S_1 = S_2$ for some~$\lambda \in \C$ 
    by Lemma~\ref{lem:adv}.
    Since~$V\neq 0$,
    there is $x \in \mathscr{H}$ with $Vx\neq 0$.
    Then~$S_1 V x = V' x = S_2 V x = \lambda S_1 V x$,
    and so~$\lambda=1$. Thus~$S_1 = S_2$, as desired.

\noindent \emph{(Existence)}\ 
Note that for all~$a_1, \ldots, a_n \in \mathscr{A}$
        and~$x_1, \ldots, x_n \in \mathscr{H}$,
    we have
\begin{equation*}
    \bigl\|\  \sum_i \pi(a_i) Vx_i \ \bigr\|^2
        \ = \  \sum_{i,j} \left< V^*\pi(a_j^*a_i) Vx_i, x_j  \right> 
        \ =\ \sum_{i,j} \left< \varphi(a_j^*a_i)x_i, x_j  \right> 
        \ =\ \bigl\|\  \sum_i \pi'(a_i) V'x_i \ \bigr\|^2.
\end{equation*}
    Hence there is a unique
isometry~$S\colon \mathscr{K} \to \mathscr{K}'$
        such that~$S \pi(a)Vx = \pi'(a)V'x$
	for all~$a\in\mathscr{A}$
		and~$x\in \mathscr{H}$.
Since $SVx=S\pi(1)Vx = \pi'(1)V'x = V'x$
for all~$x\in \mathscr{H}$,
we have~$SV=V'$. 
Further, for all~$a, a_1, \ldots, a_n \in \mathscr{A}$,
        and~$x_1, \ldots, x_n \in \mathscr{H}$, we have
\begin{equation*}
S \pi(a)  \sum_i \pi(a_i) Vx_i 
= \sum_i S \pi(a a_i) Vx_i 
= \sum_i \pi'(a a_i) V'x_i 
= \pi'(a) \sum_i \pi'(a_i) V'x_i 
= \pi'(a) S \sum_i \pi(a_i) Vx_i.
\end{equation*}
Since the linear span 
of~$\pi(\mathscr{A})V\mathscr{H}$ is dense in~$\mathscr{K}$,
we get~$S\pi(a) = \pi'(a) S$.
Note that~$S^*S= 1$, because~$S$ is an isometry.
Thus~$S^* \pi'(a) S
                = S^* S \pi(a)
                = \pi(a)$,
and so~$\Ad_S \circ \pi' = \pi$.
\end{proof}

\begin{thm}\label{thm:stinespring-paschke}
Let $(\mathscr{K}, \pi, V)$ 
be a minimal normal Stinespring dilation
of an NCP-map $\varphi\colon \mathscr{A}\to B(\mathscr{H})$.
    Then $\xymatrix{ \mathscr{A} \ar[r]|-\pi & B(\mathscr{K})
        \ar[rr]|-{ V(\,\cdot\,) V^*} && B(\mathscr{H}) }$
    is a Paschke dilation of~$\varphi$.

\end{thm}
\parpic[r]{
\xymatrix{
\mathscr{A} \ar@/_/[rdd]_{\varrho'}\ar[rd]^\pi \ar[rr]^\varphi
&& B(\mathscr{H}) \\
& B(\mathscr{K}) \ar[ru]^{\Ad_V} \\
& \mathscr{P}' \ar[ruu]|{f'} \ar[r]_{\pi'}
& B(\mathscr{K}') \ar[uu]_{\Ad_{V'}} \ar[lu]|{\Ad_S}
}
}
\begin{proof}
Let~$\mathscr{P}'$ be a von Neumann algebra.
Let
$\varrho' \colon \mathscr{A} \to \mathscr{P}'$
be
a NMIU-map,
and
$f'\colon \mathscr{P}' \to B(\mathscr{K})$
an NCP-map
with $f' \after \varrho' = \varphi$.
We must show that there is a unique
NCP-map~$\sigma\colon \mathscr{P}'\to B(\mathscr{K})$
with~$\sigma \after \varrho' = \pi$
and~$\Ad_V \after \sigma = f'$.
The uniqueness of~$\sigma$ 
follows by the same reasoning 
we used to show that~$\Ad_{S_1}=\Ad_{S_2}$
in Proposition~\ref{prop:stinespring-spatial}.
To show such $\sigma$ exists,
let~$(\mathscr{K}', \pi', V')$
be a minimal normal Stinespring dilation of~$f'$.
Note that~$(\mathscr{K}', \pi' \after \varrho', V')$
is a normal Stinespring dilation of~$\varphi$.
Thus, by Proposition~\ref{prop:stinespring-spatial},
there is a (unique) isometry~$S\colon \mathscr{K} \to \mathscr{K}'$
such that~$SV=V'$ and~$\Ad_S \after \pi' \after \varrho' = \pi$.
Define~$\sigma \equiv \Ad_S \after \pi'$.
Clearly~$\sigma \after \varrho' = 
            \Ad_s \after \pi' \after \varrho' = \pi$
and~$\Ad_V \after \sigma = \Ad_V \after \Ad_S \after \pi'
=  \Ad_{V'} \after \pi' = f'$, as desired.
\end{proof}

If we combine Theorem~\ref{thm:stinespring-paschke} with 
Lemma~\ref{lem:dilationsareisomorphic}
we get the following.
\begin{cor}
Let~$(\mathscr{K}, \pi,V)$ be a minimal 
normal Stinespring dilation
of an NCP-map $\varphi\colon \mathscr{A}\to B(\mathscr{H})$,
and let~$\xymatrix{ \mathscr{A} \ar[r]|-\varrho & \mathscr{P}
        \ar[r]|-{f} & B(\mathscr{H}) }$
	be a Paschke dilation of~$\varphi$.
  Then there is a unique
   NMIU-isomorphism $\vartheta\colon B(\mathscr{K}) \to \mathscr{P}$
        with~$\varrho = \vartheta \after \pi$ and $f \after \vartheta = \Ad_V$.
\end{cor}

\section{Existence of the Paschke Dilation}
We will show
that every  $NCP$-map~$\varphi$
between von Neumann algebras
has a Paschke dilation,
see Theorem~\ref{thm:paschke}.
For this we employ the theory of
self-dual
Hilbert $\mathscr{B}$-modules
--- developed by Paschke ---
which are,
roughly speaking,
 Hilbert spaces
in which the field of complex numbers
has been replaced by  a von Neumann algebra~$\mathscr{B}$.
Nowadays,
the more general
(not necessarily self-dual)
Hilbert $\mathscr{B}$-modules, where $\mathscr{B}$ is a $C^*$-algebra,
have become more prominent,
and so it seems appropriate
to point out from the get--go
that both
self-duality
and the fact that~$\mathscr{B}$ is a von Neumann algebra
(also a type of self-duality, by the way)
seem to be essential in 
the proof of Theorem~\ref{thm:paschke}.

We review the
definitions and results
we need 
from the theory of self-dual Hilbert $\mathscr{B}$-modules.
\begin{overview}
\label{ov:hmod}
Let~$\mathscr{B}$
be a von Neumann algebra.
\begin{enumerate}
\item
A \keyword{pre-Hilbert $\mathscr{B}$-module}~$X$
(see Def.~2.1 of~\cite{bew154}\footnote{%
 Although in~\cite{bew154}
$\left<x,y\right>$
is linear in~$x$ and anti-linear in~$y$,
we have chosen to adopt the now dominant
convention
that $\left<x,y\right>$ is anti-linear in~$x$
and linear in~$y$.})
is a right $\mathscr{B}$-module
equipped
with a \textbf{$\mathscr{B}$-valued inner product},
that is,
a map $\left<\,\cdot\,,\,\cdot\,\right>\colon X\times X\to\mathscr{B}$
such that,
for all~$x,y,y'\in X$ and~$b\in \mathscr{B}$,
\begin{multicols}{2}
\begin{enumerate}
\item
$\left<x,(y+y')b\right> = \left<x,y\right>b+\left<x,y'\right>b$;
\item
$\left<x,x\right>\geq 0$
	and $\left<x,x\right>=0$ iff $x=0$;
\item
$\left<x,y\right>^*=\left<y,x\right>$.
\end{enumerate}
\end{multicols}
\item
A \keyword{Hilbert $\mathscr{B}$-module}
(see Def.~2.4 of~\cite{bew154})
is a pre-Hilbert $\mathscr{B}$-module~$X$
which is complete
with respect to the norm $\|\,\cdot\,\|$
on~$X$ given by $\|x\|=\smash{\sqrt{\|\left<x,x\right>\|}}$.

\item
A Hilbert $\mathscr{B}$-module~$X$
is \keyword{self-dual}
(see \S3 of~\cite{bew154})
if every bounded module map $\tau \colon X\to\mathscr{B}$
is of the form $\tau=\left<x,\,\cdot\,\right>$ for some~$x\in X$.

\item
Self-duality is essential for the following result.
Let~$T\colon X\to Y$
be a bounded module map
between Hilbert $\mathscr{B}$-modules.
If~$X$ is self-dual,
then it is `adjointable';
that is: there is a
unique bounded module map
$T^*\colon Y\to X$,
called the \keyword{adjoint} of~$T$,
with $\left<Tx,y\right> = \left<x,T^*y\right>$
for all~$x\in X$ and~$y\in Y$
(see Prop.~3.4 of~\cite{bew154}).

\item
\label{hmod:dual}
Let~$X$ be a pre-Hilbert $\mathscr{B}$-module.
One can extend~$X$ to a self-dual Hilbert $\mathscr{B}$-module as follows.
The set~$X'$ of
bounded module maps from~$\tau \colon X\to \mathscr{B}$
is called the \keyword{dual} of~$X$,
and is a $\mathscr{B}$-module
via $(\tau \cdot b )(x) = b^*\cdot \tau(x)$
(see line 4 of p.~450 of~\cite{bew154}).
Note that~$X$
sits inside~$X'$
via the  
injective module map
$x\mapsto \hat{x}\equiv \left<x,\,\cdot\,\right>$.
In fact, $X'$ can
be  equipped with an $\mathscr{B}$-valued inner product
that makes~$X'$ into a self-dual
Hilbert $\mathscr{B}$-module
with $\left<\tau,\hat{x}\right>=\tau(x)$
for all~$\tau\in X'$ and~$x\in X$
(see Thm.~3.2 of~\cite{bew154}).

\item
\label{hmod:extension-to-dual}
Any bounded module map $T\colon X\to Y$
between pre-Hilbert $\mathscr{B}$-modules
has a unique extension
to a bounded module map $\tilde T\colon X'\to Y'$
(see Prop.~3.6 of~\cite{bew154}).

It follows that any bounded module map $T\colon X\to Y$
from a pre-Hilbert $\mathscr{B}$-module
into a self-dual Hilbert $\mathscr{B}$-module
has a unique extension $\overline{T}\colon X'\to Y$.

\item
\label{hmod:ax}
Let~$X$ be a self-dual Hilbert $\mathscr{B}$-module.
The set~$\A(X)$ of bounded
module maps on~$X$
forms a von Neumann algebra
(see Prop.~3.10 of~\cite{bew154}).%
\footnote{The superscript~$a$ in~$\A(X)$ stands
for adjointable, which is automatic for
bounded module-maps on a self-dual~$X$.}
Addition and scalar multiplication
are computed coordinate-wise in~$\A(X)$;
multiplication
is given by composition,
and involution is the adjoint.

An element~$t$ of~$\A(X)$ is positive
iff $\left<x,x\right>\geq 0$
for all~$x\in X$
(see Lem~4.1 of~\cite{lance:hilbert}).
If~$X$ happens to be the dual of a pre-Hilbert $\mathscr{B}$-module~$X_0$,
then we even have $t\geq0$ iff $\left<x,tx\right>\geq 0$
for all~$x\in X_0$.

It follows that $T^*T$ is positive in~$\A(X)$
for any bounded module
map $T\colon X\to Y$.

We will also need the fact
that  $\left<x,(\,\cdot\,)x\right>\colon \A(X)\to \mathscr{B}$
is normal for every~$x\in X$,
which
follows from the observation
 that $f(\left<x,(\,\cdot\,)x\right>)\colon \A(X)\to \mathbb{C}$
is normal for every positive normal map $f\colon \mathscr{B}\to\mathbb{C}$,
which in turn follows form the description of the predual
of~$\A(X)$ in Proposition~3.10
of~\cite{bew154}.
\end{enumerate}
\end{overview}

\begin{dfn}
Let $\varphi\colon \mathscr{A}\to\mathscr{B}$
be an NCP-map between von Neumann algebras.
A complex bilinear map
of the form $B\colon \mathscr{A}\times \mathscr{B}\to X$,
where $X$ is a self-dual Hilbert $\mathscr{B}$-module,
is called \keyword{$\varphi$-compatible}
if, there is $r>0$ such that, for all~$a_1,\dotsc,a_n\in\mathscr{A}$
and $b_1,\dotsc,b_n\in \mathscr{B}$,
\begin{equation}
\label{eq:phi-compatible}
\bigl\|\,\sum_{i} B(a_i,b_i)\,\bigr\|^2
\ \leq\ 
r\cdot
\bigl\|\,
\sum_{i,j} b_i^*\varphi(a_i^*a_j)b_j 
\,\bigr\|,
\end{equation}
and $B(a,b_1)b_2 = B(a,b_1b_2)$
for all~$a\in\mathscr{A}$
	and $b_1,b_2\in\mathscr{B}$.
\end{dfn}

\begin{thm}
\label{thm:paschke}
Let $\varphi\colon \mathscr{A}\to\mathscr{B}$
be an NCP-map between von Neumann algebras.
\begin{enumerate}
\item
\label{paschke-1}
There is a self-dual Hilbert 
$\mathscr{B}$-module $\mathscr{A}\otimes_\varphi \mathscr{B}$
and a $\varphi$-compatible
bilinear map 
\begin{equation*}
\otimes\colon \mathscr{A}\times \mathscr{B}
\to \mathscr{A}\otimes_\varphi \mathscr{B}
\end{equation*}
such that for every $\varphi$-compatible bilinear map 
$B\colon \mathscr{A}\times \mathscr{B}\to Y$
there is a unique bounded module
map $T\colon \mathscr{A}\otimes_\varphi\mathscr{B}\longrightarrow Y$
such that $T(a\otimes b)=B(a,b)$ for all~$a\in \mathscr{A}$
and $b\in\mathscr{B}$.
\item
\label{paschke-2}
For every~$a_0\in\mathscr{A}$
there is a unique bounded module map
$\varrho(a_0)$ on $\mathscr{A}\otimes_\varphi\mathscr{B}$
given by 
\begin{equation*}
\varrho(a_0)(a\otimes b) \ =\  (a_0a)\otimes b,
\end{equation*}
and the assignment $a\mapsto \varrho(a)$
yields an NMIU-map 
$\varrho\colon \mathscr{A}\to \A(\mathscr{A}\otimes_\varphi \mathscr{B})$.

\item
\label{paschke-3}
The assignment $T\mapsto \left<1\otimes 1,T(1\otimes 1)\right>$
gives an NCP-map $f\colon \A(\mathscr{A}\otimes_\varphi\mathscr{B})
\longrightarrow \mathscr{B}$.

\item
\label{paschke-4}
$\xymatrix{\mathscr{A} 
\ar[r]|-\varrho
&
\A(\mathscr{A}\otimes_\varphi\mathscr{B})
\ar[r]|-{f}
&
\mathscr{B}}$
is a Paschke dilation of~$\varphi$.
\end{enumerate}
\end{thm}
\begin{proof}
\emph{(\ref{paschke-1})}\ 
To construct $\mathscr{A}\otimes_\varphi\mathscr{B}$,
we follow the lines of Theorem~5.2 of~\cite{bew154},
and start with the algebraic tensor product,
$\mathscr{A}\odot\mathscr{B}$,
whose elements are finite sums of the form $\sum_i a_i\otimes b_i$,
and which is a right $\mathscr{B}$-module
via $(\sum_ia_i\otimes b_i)\beta = \sum_ia_i\otimes(b_i\beta)$.
If we define $[\,\cdot\,,\,\cdot\,]$
on~$\mathscr{A}\odot\mathscr{B}$ by
\begin{equation*}
\bigl[\ \sum_i a_i\otimes b_i,\ \sum_j \alpha_j \otimes \beta_j\ \bigr]
\ = \ \sum_{i,j} b_i^* \varphi(a_i^*\alpha_j)\beta_j,
\end{equation*}
we get a $\mathscr{B}$-valued semi-inner product
on~$\mathscr{A}\odot\mathscr{B}$,
and a (proper) $\mathscr{B}$-valued inner product
on the quotient $X_0 = (\mathscr{A}\odot\mathscr{B})/N$,
where~$N=\{x\in\mathscr{A}\odot\mathscr{B}\colon [x,x]=0\}$,
so that~$X_0$
is a pre-Hilbert $\mathscr{B}$-module.
(That~$N$ is a submodule
of~$\mathscr{A}\odot\mathscr{B}$
is not entirely obvious,
see Remark~2.2 of~\cite{bew154}.)
Now,
define~$\mathscr{A}\otimes_\varphi \mathscr{B} := X_0'$
(where $X_0'$ is the dual of~$X_0$ from 
Overview~\ref{ov:hmod}\eqref{hmod:dual}),
and let~$\otimes \colon \mathscr{A}\times \mathscr{B}
\to \mathscr{A}\otimes_\varphi\mathscr{B}$
be given by $a\otimes b = \smash{\widehat{a\otimes b + N}}$.
Then~$\mathscr{A}\otimes_\varphi \mathscr{B}$
is a self-dual Hilbert $\mathscr{B}$-module,
and~$\otimes$
is a $\varphi$-compatible bilinear map.

Let~$B\colon \mathscr{A}\times\mathscr{B}\to Y$
be a $\varphi$-compatible bilinear map
to some self-dual Hilbert $\mathscr{B}$-module~$Y$.
We must show that there is a unique bounded
module map $T\colon\mathscr{A}\otimes_\varphi\mathscr{B}\to Y$
such that~$T(a\otimes b) = B(a,b)$.
By Overview~\ref{ov:hmod}\eqref{hmod:extension-to-dual},
it suffices to show that there is a
unique bounded module map $T\colon X_0\to Y$
with $T(a\otimes b) = B(a,b)$.
Since $\mathscr{A}\odot\mathscr{B}$
is generated by elements of the form~$a\otimes b$,
uniqueness is obvious.
Concerning existence,
there is a (unique) linear map $S\colon \mathscr{A} \odot \mathscr{B}
\to X$ with $S(a\otimes b) = B(a,b)$
by the universal property of the algebraic tensor product.
Note that the  kernel of~$S$ contains~$N$,
because if $x=\sum_i a_i\otimes b_i$
is from~$N$,
then $[x,x]=\sum_{i,j} b_i^*\varphi(a_i^*a_j)b_j=0$,
and so~$\|S(x)\| \equiv \|\sum_iB(a_i,b_i)\|=0$
by Equation~\eqref{eq:phi-compatible}. 
Thus there is a unique linear map $T\colon X_0\to Y$
with $T(a\otimes b) = B(a,b)$.
By Equation~\eqref{eq:phi-compatible}, $T$ is bounded.
Finally, 
since $B(a,b\beta)=B(a,b)\beta$,
it is easy to see that~$S$ and~$T$ are module maps.

\emph{(\ref{paschke-2})}\ 
Let~$a_0\in \mathscr{A}$
be given.
To obtain the bounded module map~$\varrho(a_0)$
on~$\mathscr{A}\otimes_\varphi\mathscr{B}$,
it suffices to show the bilinear
map $B\colon \mathscr{A}\times \mathscr{B}\to 
\mathscr{A}\otimes_\varphi\mathscr{B}$
given by $B(a,b)=(a_0a)\otimes b$
is $\varphi$-compatible.
It is easy to see that $B(a,b)\beta= B(a,b\beta)$.
Concerning Equation~\eqref{eq:phi-compatible},
let $a_1,\dotsc,a_n\in\mathscr{A}$
and~$b_1,\dotsc,b_n\in\mathscr{B}$
be given.
Then we have,
writing $a$ for the row vector $(a_1 \dotsb a_n)$,
$b$ for the column vector $(b_1 \dotsb b_n)$,
\begin{alignat*}{3}
\bigl\|\, \sum_i B(a_i,b_i)\,\bigr\|^2
\ &=\ 
\bigl\|\, \sum_i (a_0a_i)\otimes b_i \,\bigr\|^2 
\ =\ 
\bigl\|\, \sum_{i,j} b_i^* \varphi(a_i^* a_0^*a_0 a_j) b_j\,\bigr\| 
\ =\ 
\|b^* (M_n\varphi) (a^* a_0^*a_0 a) b\| \\
\ &\leq \ 
\|a_0^*a_0\|\,\cdot\,
\|\,b^* (M_n\varphi)(a^* a) b \,\|   
\ =\ \|a_0\|^2\,\cdot\, \bigl\|\,\sum_{i,j} b_i^*\varphi(a_i^*a_j)b_j
\,\bigr\|.
\end{alignat*}
Thus~$B$ is $\varphi$-compatible,
and so there is a unique bounded module map 
$\varrho(a_0)\colon \mathscr{A} \otimes_\varphi \mathscr{B}
\longrightarrow \mathscr{A}\otimes_\varphi\mathscr{B}$
with $\varrho(a_0)(a\otimes b) = (a_0a)\otimes b$.
Since it is easy to see
that $a_0\mapsto \varrho(a_0)$
gives a multiplicative
involutive unital linear
map~$\varrho\colon\mathscr{A}\to \A(\mathscr{A}\otimes_\varphi\mathscr{B})$,
the only thing left to prove is that~$\varrho$ is normal.

Let~$D$ be a bounded directed set of self-adjoint elements of~$\mathscr{A}$.
To show that~$\varrho$ is normal,
we must prove that~$\varrho(\sup D) = \sup_{d\in D}\varrho(d)$.
 It suffices to show
that $\left<x,\varrho(\sup D)x\right>
= \left<x,\sup_{d\in D}\varrho(d)x\right>$
for all~$x\in X_0$.
Let~$x\in X_0$ be given and write $x=\sum_i a_i\otimes b_i$,
where $a_1,\dotsc,a_n\in\mathscr{A}$
and~$b_1,\dotsc,b_n\in\mathscr{A}$.
Then, if $a$ stands for the row vector $(a_1 \dotsb a_n)$
and~$b$ is the column vector~$(b_1,\dotsc,b_n)$,
we have
\begin{equation*}
\left<x,\smash{\varrho(\sup D)}x\right>
\ =\ b^* (M_n\varphi)(a^*\,\sup D\, a) b
\ = \ \sup_{d\in D}\, b^* (M_n\varphi)(a^*da) b
\ = \ \sup_{d\in D}\left<x,\varrho(d)x\right>
\ = \ \bigl<x,\sup_{d\in D} \varrho(d)x\bigr>,
\end{equation*}
where we used that
$b^*(M_n\varphi)(a^*(\,\cdot\,) a)b$
and $\left<x,(\,\cdot\,)x\right>$
are normal.
Thus~$\varrho$ is normal.

\emph{(\ref{paschke-3})}\ 
Write $e=1\otimes 1$.
We already know that~$\left<e,(\,\cdot\,)e\right>$
is normal,
and since for~$t_1,\dotsc,t_n\in \A(\mathscr{A}\otimes_\varphi \mathscr{B})$
and~$b_1,\dotsc,b_n\in \mathscr{B}$,
we have
$\sum_{i,j} b_i^*\left< e ,t_i^*t_j e \right>b_j
=
\left< \sum_i t_i e b_i,\sum_j t_j e b_j\right> \geq 0$,
we see by  Remark~5.1 of~\cite{bew154}, that 
$\left<e,(\,\cdot\,)e\right>$
is completely positive.

\emph{(\ref{paschke-4})}\ 
To begin,
note that $(f\circ \varrho)(a)
= \left<1\otimes 1,a\otimes 1\right>=\varphi(a)$
for all~$a\in \mathscr{A}$,
and so~$\varphi = f\circ \varrho$.

Suppose that~$\varphi$
factors
as $\xymatrix{\mathscr{A} \ar[r]|{\varrho'} &
\mathscr{P}' \ar[r]|{f'}  &
\mathscr{B}}$,
where~$\mathscr{P}'$
is a von Neumann algebra,
$\varrho'$ is an NMIU-map,
and~$f'$ is an NCP-map.
We must show  that there is a unique NCP-map
$\sigma\colon \mathscr{P}'\to \A(\mathscr{A}\otimes_\varphi \mathscr{B})$
with $f\circ \sigma = f'$ and~$\sigma\circ\varrho' = \varrho$.

\emph{(Uniqueness)}\ 
Let~$\sigma_1,\sigma_2\colon \mathscr{P}\to \A(\mathscr{A}\otimes_\varphi 
\mathscr{B})$
be NCP-maps with $f\circ\sigma_k = f' $ and 
$\sigma_k \circ \varrho' = \varrho$.
We must show that~$\sigma_1=\sigma_2$.
Let~$c\in \mathscr{P}'$
and~$x\in X_0$ be given.
It suffices to prove that~$\left<x,\sigma_1(c)x\right>
= \left<x,\sigma_2(c)x\right>$
(see Overview~\ref{ov:hmod}\eqref{hmod:ax}).
Write~$x=\sum_i a_i\otimes b_i$
where $a_i \in\mathscr{A}$,~$b_i\in\mathscr{B}$.
Then, $a_i\otimes b_i = \varrho(a_i)(1\otimes 1)b_i$,
and so
\begin{equation}
\label{eq:sigma}
\begin{alignedat}{3}
\left< x,\sigma_i(c)x\right>
\ &=\ 
\sum_{i,j} b_i^* f(\, \varrho(a_i^*)
	\sigma_i(c)\varrho(a_j) \,) b_j 
	\qquad&&\text{by an easy computation} \\
\ &=\ 
\sum_{i,j} b_i^* 
	f(\sigma_i(\varrho'(a_i^*) c\varrho'(a_j))) b_j 
	\qquad&&\text{by Lemma~\ref{lem:choi}} \\
	\ &=\ 
\sum_{i,j} b_i^* 
	f'(\varrho'(a_i^*) c\varrho'(a_j)) b_j 
	\qquad &&\text{since $f'= f\circ \sigma_i$}.
\end{alignedat}
\end{equation}
Hence $\left<x,\sigma_1(x)\right>
= \left<x,\sigma_2(x)\right>$,
and so~$\sigma_1 = \sigma_2$.

\emph{(Existence)}\ 
Recall that each self-adjoint 
bounded operator~$A$ on Hilbert space~$\mathscr{H}$
gives a bounded quadratic form~$x\mapsto \left<x,Ax\right>$,
that
every quadratic form arises in this way,
and that the operator~$A$ can be reconstructed from
its quadratic form.
One can develop a similar correspondence
in the case of Hilbert $\mathscr{B}$-modules,
which can be used to define~$\sigma$
from Equation~\eqref{eq:sigma}.
We will, however,
give a shorter proof
of the existence
of~$\sigma$, which was suggested to us by Michael Skeide.

The trick is to see that the construction that gave us
$\mathscr{A}\otimes_\varphi\mathscr{B}$
may also be applied to~$f'\colon \mathscr{P}'\to \mathscr{B}$
yielding maps
$\xymatrix{\mathscr{P}'
\ar[r]|-{\varrho''}
&
\A(\mathscr{P}'\otimes_{f'}\mathscr{B})
\ar[r]|-{f''}
&
\mathscr{B}
}$.
It suffices to find an NCP-map 
$\sigma'\colon \A(\mathscr{P'}\otimes_{f'}\mathscr{B})
\longrightarrow \A(\mathscr{A}\otimes_\varphi\mathscr{B})$
with $f'' = f\circ \sigma'$
and~$\varrho = \sigma'\circ \varrho''\circ \varrho'$
for then $\sigma = \sigma'\circ \varrho''$
will have the desired properties.

Let~$S\colon \mathscr{A}\otimes_\varphi \mathscr{B}
\longrightarrow \mathscr{P}'\otimes_{f'}\mathscr{B}$
be the bounded module map
given by~$S(a\otimes b) = \varrho'(a)\otimes b$,
which exists by part~\ref{paschke-1},
because a straightforward
computation shows
that $(a,b)\mapsto \varrho'(a)\otimes b$
gives a $\varphi$-compatible bilinear map
$\mathscr{A}\times \mathscr{B}\longrightarrow
\mathscr{P}'\otimes_{f'}\mathscr{B}$.
We claim that $\sigma' = S^*(\,\cdot\,)S$
fits the bill.

Let us begin by proving that~$\sigma'(1)\equiv S^* S  =1$.
Let~$x\in X_0$.
It suffices to show that $\left<x,S^*Sx\right>=\left<x,x\right>$.
Writing $x\equiv \sum_i a_i \otimes b_i$,
we have $\left< x,S^* Sx\right>
= \left<Sx,Sx\right>
= \sum_{i,j} b_i^* f'(\varrho'(a_i^*)\varrho'(a_j))b_j 
= \sum_{i,j} b_i^* \varphi(a_i^*a_j) b_j
= \left< x,x\right>$,
because~$\varrho'$ is multiplicative
and~$f'\circ\varrho'=\varphi$.
Thus~$S^*S=1$.

Note that~$\sigma'$ is completely positive,
because for all~$s_1,\dotsc,s_n \in
\A(\mathscr{P}'\otimes_{f'}\mathscr{B})$
and $t_1,\dotsc,t_n
\in \A(\mathscr{A}\otimes_\varphi \mathscr{B})$,
we have $\sum_{i,j} t_i^* \sigma'(s_i^*s_j)t_j
=  (\sum_i s_iSt_i)^*(\sum_j s_jSt_j) \geq 0$
(see Remark~5.1 of~\cite{bew154}).

Let~$x\in \mathscr{A}\otimes_\varphi\mathscr{B}$
be given.
Note that $\left<x,\sigma'(\,\cdot\,)x\right>
= \left<Sx,(\,\cdot\,)Sx\right>$ is normal.
From this it follows
that~$\sigma'$ is normal
(in the same way we proved that~$\varrho$ is normal
in~\eqref{paschke-2}).

Since from~$S(1\otimes 1)=\varrho(1)\otimes 1 = 1\otimes 1$
it swiftly follows that~$f\circ \sigma' = f''$,
the only thing left to show
is that~$\varrho=\sigma'\circ \varrho''\circ \varrho'$.
Let~$a,a_0\in \mathscr{A}$
and~$b\in \mathscr{B}$
be given.
By point~\eqref{paschke-1},
it suffices to show that $\varrho(a_0)(a\otimes b)
= \sigma'(\varrho''(\varrho'(a_0)))(a\otimes b)$.
Unfolding gives
$\sigma'(\varrho''(\varrho'(a_0)))(a\otimes b) 
= S^* (\varrho'(a_0a)\otimes b) = S^*S( \varrho(a_0)( a\otimes b))$,
but we already saw that~$S^*S=1$,
and so we are done.
\end{proof}

\section{Pure maps}
Schr\"odinger's equation is invariant under the reversal of time,
and so any isolated purely quantum mechanical process is invertible.
However, NCP-maps include non-invertible processes
such as measurement and discarding.
However, not all is lost, for a broad class of processes
is pure enough to be `reversed', e.g.~$(\Ad_V)^\dagger = \Ad_{V^*}$.
In fact, the Stinespring dilation theorem
states that every NCP-map into~$B(\mathscr{H})$
factors as a reversible~$\Ad_V$,
after a (possibly) non-reversible NMIU-map.

In this section we study two seemingly unrelated
definitions of pure (i.e.~reversible)
for arbitrary NCP-maps.
The first is a direct generalization of~$\Ad_V$,
and the second uses the Paschke dilation.
Both definitions turn out to be equivalent.

Before we continue, let's rule out two alternative definitions of pure.
\begin{inparaenum}
\item
Recall a state~$\varphi\colon \mathscr{A} \to \C$
is called pure, if is an extreme point among all states.
It does not make sense to define
an NCP-map to be pure if it is extreme,
because every NMIU-map is extreme (among the unital NCP-maps).
\item
Inspired by the GNS-correspondence between pure states
and irreducible representations,
St\o rmer defines a map~$\varphi\colon \mathscr{A} \to B(\mathscr{H})$
to be pure if the only maps below~$\varphi$
in the completely positive order are scalar multiples of~$\varphi$.
One can show that the St\o rmer pure NCP-maps
between~$B(\mathscr{H}) \to B(\mathscr{K})$
are exactly those of the form~$\Ad_V$, see Prop.~\ref{prop:stormerpure}.
In a non-factor
every central element gives a different completely positive map
below the identity and so if one generalizes
St\o rmer's definition to arbitrary NCP-maps,
the identity need not be pure.
\end{inparaenum}

Now, let us sketch our first definition of pure.
Consider~$V \colon \mathscr{K} \to\mathscr{H}$.
By polar decomposition
we may factor~$V = UA$,
where~$A\colon \mathscr{K} \to r(A) \mathscr{H}$ is a positive map
and~$U\colon r(A)\mathscr{H} \to \mathscr{K}$ is an isometry
and so~$\Ad_V = \Ad_A \after \Ad_U$.
The pure maps~$\Ad_A$ and~$\Ad_U$ are of a particularly simple form.
We will see they admit a dual universal property,
the first is (up to scaling) a \emph{compression}
and the second a \emph{corner} (Def.~\ref{dfn:ww16}).
This allows us to generalize
the notion of pure to arbitrary NCP-maps (Def.~\ref{dfn:pure}).
We show maps on the right-hand side of a Paschke dilation are pure.
Then we will show
the main result of the section:
an NCP-map is pure if and only if
the map on the left-hand side of its Paschke dilation is surjective.
As a corollary,
we show both the left- and right-hand side of a Paschke dilation
are extreme among the maps with the same value on~$1$.

\begin{dfn}\label{dfn:ww16}
Let~$a$ be an element of a von Neumann algebra~$\mathscr{A}$
with~$0\leq a\leq 1$.
\begin{enumerate}
    \item
    The least projection above~$a$,
        we call the \keyword{support projection} of~$a$,
        and we denote it as~$\ceil{a}$. Its de Morgan dual
    $\floor{a} \equiv 1-\ceil{1-a}$ is the greatest projection below~$a$.
        For any projection~$p$, write~$C_p$ for the \keyword{central carrier},
            that is: the least central projection above~$p$
	    (see e.g.~\cite[Def.~5.5.1]{kr}).

    \item
    For an NCP-map~$\varphi\colon \mathscr{A} \to \mathscr{B}$,
    we write~$\car \varphi$
        for the \keyword{carrier} of~$\varphi$, 
		the least projection of~$\mathscr{A}$
        such that $\varphi(\car \varphi) = \varphi(1)$.
        The map~$\varphi$ is said to be \keyword{faithful}
            if~$\car \varphi=1$.
        Equivalently, $\varphi$ is faithful
            if~$\varphi(a^*a) =0$ implies~$a^*a =0$ for all~$a \in \mathscr{A}$.
    \item
    We call the map~$h_a\colon \mathscr{A} \to \floor{a} \mathscr{A} \floor{a}$
            given by~$b \mapsto \floor{a} b \floor{a}$,
        the \keyword{standard corner} of~$a$.

    \item
    We call the map~$c_a\colon \ceil{a} \mathscr{A} \ceil{a} \to \mathscr{A}$
        given by~$b \mapsto \sqrt{a} b \sqrt{a}$,
        the \keyword{standard compression} of~$a$.

    \item
    A contractive NCP-map~$h \colon \mathscr{A} \to \mathscr{B}$
        is said to be a \keyword{corner}
        for an~$a \in [0,1]_{\mathscr{A}}$
    if~$h(a)=h(1)$ and for
        every (other) contractive NCP-map~$f \colon \mathscr{A} \to \mathscr{C}$
            with~$f(a)=f(1)$,
        there is a unique~$f'\colon \mathscr{B} \to \mathscr{C}$
            with~$f = f' \after h$.
    \item
    A contractive NCP-map~$c \colon \mathscr{B} \to \mathscr{A}$
        is said to be a \keyword{compression}
        for an~$a \in [0,1]_{\mathscr{A}}$
    if~$c(1)=a$ and for
        every (other) contractive NCP-map~$g \colon \mathscr{C} \to \mathscr{A}$
            with~$g(1)\leq a$,
        there is a unique~$g'\colon \mathscr{C} \to \mathscr{B}$
            with~$g = c \after g'$.
\end{enumerate}
\end{dfn}
\begin{prop}\label{prop:ww16}
Let~$\mathscr{A}$ be any von Neumann algebra
with effect~$a \in [0,1]_{\mathscr{A}}$.
\begin{enumerate}
    \item
    The standard corner~$h_a$ is a corner for~$a$ and
    the standard compression~$c_a$ is a compression for~$a$.
    \item
    Corners are surjective.  Compressions are injective.
    Restricted to self-adjoint elements,
        compressions are order-embeddings.
    \item
    Every corner~$h$ for~$a$ is of the form~$h=\vartheta \after h_a$
        for some isomorphism~$\vartheta$.
    Every compression~$c$ for~$a$ is of the form~$c=c_a \after \vartheta$
        for some isomorphism~$\vartheta$.
    \item
    Assume~$\varphi\colon \mathscr{A} \to \mathscr{B}$
        is a contractive NCP-map.
    Let~$\varphi'$ denote the unique contractive NCP-map
        such that~$c_{\varphi(1)} \after \varphi' = \varphi$
    and~$\varphi''$ the unique contractive NCP-map
        with~$\varphi'' \after h_{\car \varphi} = \varphi$.
    Then~$\varphi'$ is unital and~$\varphi''$ is faithful.
    \item
    With~$\varphi$ as above,
    there is a unique NCP-map~$\varphi_\ft\colon (\car{\varphi }) \mathscr{A}(\car\varphi)
            \to \ceil{\varphi(1)}\mathscr{B}\ceil{\varphi(1)}$
    such that~$c_{\varphi(1)}\after \varphi_\ft \after h_{\car\varphi} = \varphi$. The map~$\varphi_\ft$ is faithful and unital.
\end{enumerate}
\end{prop}
\begin{proof}[Proof of Proposition~\ref{prop:ww16}]
The proposition is true in greater generality.\cite{effectus}
We give a direct proof.
\begin{enumerate}
\item
Proven in~\cite[Prop.~5 \& 6]{ww16}.
\item
Let~$c\colon \mathscr{A} \to \mathscr{B}$ be any compression.
Assume~$a,a' \in \mathscr{A}$
are self-adjoint elements with~$c(a) \leq c(a')$.
Without loss of generality,
we may assume~$a,a' \geq 0$
as~$\|a\| + a, \|a\|+a' \geq 0$.
Define~$p_\alpha\colon \mathscr{C} \to \mathscr{B}$
by~$p_\alpha(1) = \alpha$ for~$\alpha \in \mathscr{A}$.
Note~$0 \leq c(a'-a) \leq c(1)$
and so by the universal property of~$c$,
there is a unique map~$f'\colon \C \to \mathscr{A}$
with~$c \after f' = p_{c(a'-a)}$.
We compute
\begin{equation*}
c(p_{\frac{1}{2}a'}(1)) 
=\frac{1}{2}c(a') = \frac{c(a) + p_{c(a'-a)}(1)}{2}
            = \frac{c(a) + c(f'(1))}{2} = c(p_{\frac{1}{2}(a+f'(1))}(1))
\end{equation*}
and so by the universal property of~$c$,
we have~$p_{\frac{1}{2}a'} = p_{\frac{1}{2}(a+f'(1))}$
and so~$a \leq a'$, as desired.
As a corollary, $c$ is injective on self-adjoint elements.
It follows~$c$ is injective.  (See e.g.~\cite[Proof Lemma 4.2]{cho}).
Clearly isomorphisms and a standard corner are surjective.
Thus by point 3, every corner is surjective.

\item
A standard argument gives
us that there are mediating NCP-isomorphisms.
They are actual NMIU-isomorphisms
by e.g.~\cite[Corollary 47]{ww16}.

\item
Write~$u_\mathscr{A}\colon \C \to \mathscr{A}$
for the NCP-map $u_{\mathscr{A}}(\lambda) = \lambda\cdot 1$.
Then~$\varphi \after u_{\mathscr{A}}
    = c_{\varphi(1)} \after \varphi' \after u_{\mathscr{A}}
            = c_{\varphi(1)} \after u_{\ceil{\varphi(1)} \mathscr{B} \ceil{\varphi(1)}}$
and so by the universal property of~$c_{\varphi(1)}$,
we get~$\varphi' \after u_\mathscr{A} = u_{\ceil{\varphi(1)} \mathscr{B} \ceil{\varphi(1)}}$
and so~$1=\varphi'(1)$, as desired.
Write~$p \equiv \car \varphi$.
Now, to show~$\varphi''$ is faithful,
assume~$\varphi''(pap) =0$ for some~$pap \in [0,1]_{p\mathscr{A}p}$.
Then~$0=\varphi''(pap) = \varphi''(h_p(pap)) = \varphi(pap)$
and so~$pap \leq 1 - \car \varphi = 1-p$.
Hence~$pap = 0$, as desired.
\item
By point 2, we know $\varphi_\ft$ is unique.
Note~$c_{\varphi(1)} \after \varphi_\ft (1) = \varphi(1)$
and~$\car (\varphi_\ft \after h_{\car \varphi}) = \car \varphi$
and so $\varphi_\ft$ is unital and faithful
by the previous point. \qedhere
\end{enumerate}
\end{proof}
\noindent
In the previous Definition and Proposition we chose
to restrict ourselves to contractive maps
as a `non-contractive compression' might sound confusing.
For a non-contractive~$\varphi$,
define~$\varphi_\ft := (\frac{1}{\|\varphi\|}\cdot \varphi)_\ft$.
\begin{dfn}\label{dfn:pure}
    A NCP-map~$\varphi\colon \mathscr{A} \to \mathscr{B}$
        is said to be \keyword{pure}
        whenever~$\varphi_\ft$ is an isomorphism.
\end{dfn}
\begin{exa}
    The pure NCP-maps between~$B(\mathscr{H})$ and $B(\mathscr{K})$
        are exactly those of the form~$\Ad_V$
        for some~$V\colon \mathscr{K} \to \mathscr{H}$.
    The contractive pure maps are exactly those with~$V^*V\leq 1$.
\end{exa}
\begin{exa}
    If~$\varphi \colon \mathscr{A} \to \mathscr{B}$
        is unital and faithful,
    then~$\varphi_\ft = \varphi$.
    Thus a state is pure if and only if it is faithful.
    Also isomorphisms are pure.
\end{exa}
\begin{prop}\label{prop:compclosed}
    Corners, compressions and pure maps are closed under composition.
\end{prop}
\begin{proof}
    First we show corners are closed under composition.
    Let~$h\colon \mathscr{A} \to \mathscr{B}$ be a corner for~$a$
        and~$h' \colon \mathscr{B} \to \mathscr{C}$ be a corner for~$b$.
    There is a unique isomorphism~$\vartheta
            \colon \floor{a} \mathscr{A} \floor{a} \to \mathscr{B}$
    such that~$h = \vartheta \after h_{\floor{a}}$.
    Note~$h_{\floor{a}} \after c_{\floor{a}} = \id$
        and so~$h \after c_{\floor{a}} \after \vartheta^{-1} = \id$.
    We will show~$h' \after h$ is a corner
        for~$c_{\floor{a}} (\vartheta^{-1}(\floor{b}))$.
    To this end, assume~$g\colon \mathscr{A} \to \mathscr{D}$
        is any contractive NCP-map
        for which it
        holds~$g(c_{\floor{a}} (\vartheta^{-1}(\floor{b}))) =g(1)$.
    Clearly~$g(1) =g(c_{\floor{a}} (\vartheta^{-1}(\floor{b}))) =g(1)
                \leq g(\floor{a}) \leq g(1)$
    and so by the universal property of~$h$,
        there is a unique contractive
    NCP-map~$g'\colon \mathscr{B} \to \mathscr{D}$
        such that~$g' \after h = g$.
    Now
\begin{equation*}
g'(\floor{b})
            = g'(h(c_{\floor{a}}(\vartheta^{-1}(\floor{b}))))
            = g(c_{\floor{a}}(\vartheta^{-1}(\floor{b})))
            = g(1) = g'(h(1)) = g'(1)
\end{equation*}
and so by the universal
    property of~$h'$ there is a unique contractive
        NCP-map~$g''\colon \mathscr{C} \to \mathscr{D}$
        with~$g'' \after h' = g'$.
    Clearly~$g'' \after h' \after h = g$.
    It is easy to see~$g''$ is unique.

We continue with compressions.
Assume~$c\colon \mathscr{C} \to \mathscr{B}$
is a compression for~$b$
and~$c'\colon \mathscr{B} \to \mathscr{A}$
is a compression for~$a$.
We will show~$c' \after c$ is a compression for~$c'(b)$.
To this end, let~$g\colon \mathscr{D} \to \mathscr{A}$
be any contractive NCP-map such that~$g(1) \leq c'(b)$.
As~$g(1) \leq c'(b) \leq a$,
there is a unique~$g'\colon \mathscr{D} \to \mathscr{B}$
        with~$c' \after g' = g$.
Clearly~$c'(g'(1)) =  g(1) \leq c'(b)$.
Thus~$g'(1) \leq b$
and so there is a unique contractive
NCP-map~$g''\colon \mathscr{D} \to \mathscr{C}$
such that~$c \after g'' = g'$.
Now~$c' \after c \after g'' = g$
It is easy to see~$g''$ is unique.

\parpic[r]{$\xymatrix@C-1.8pc@R-1pc{
& \mathscr{A} \ar[rd]^{h_p}&\\
{\ceil{a}\mathscr{A}\ceil{a}} \ar[ru]^{c_a} 
    \ar[d]_{h_{\ceil{\sqrt{a}p\sqrt{a}}}}&
& p \mathscr{A} p\\
\ceil{\sqrt{a} p \sqrt{a}}
\mathscr{A}
\ceil{\sqrt{a} p \sqrt{a}}
\ar[rr]_{\Ad_u}
&&
\ceil{pap} \mathscr{A} \ceil{pap}
\ar[u]_{c_{pap}}
}$}

To show pure maps are closed under composition,
it is sufficient to show that~$h_p \after c_a$
is pure for
any von Neumann algebra~$\mathscr{A}$,
projection~$p \in \mathscr{A}$ and effect~$a \in [0,1]_{\mathscr{A}}$.
To this end, we will define a~$u$
such that the diagram on the right makes sense, commutes
and~$\Ad_u$ is an isomorphism.
First some facts.

\begin{enumerate}
\item
By polar decomposition (see e.g.~\cite[p.15]{topping}), there is
a partial isometry~$u \in \mathscr{A}$ such
that~$\sqrt{a}p = u\sqrt{pap}$ with domain~$u^*u=\ceil{pap}$
and range~$uu^*=r(\sqrt{a}p)=r(\sqrt{a}p(\sqrt{a}p)^*)=\ceil{\sqrt{a}p\sqrt{a}}$,
where with~$r(b)$ we denote the projection
onto the closed range of~$b$.
Note that adjoining by~$u$
restricts to an
isomorphism~$\Ad_u\colon
\ceil{\sqrt{a} p \sqrt{a}} \mathscr{A} \ceil{\sqrt{a} p \sqrt{a}}
\to
\ceil{pap} \mathscr{A} \ceil{pap}$.

\item
We have~$\sqrt{a} p \sqrt{a} \leq a \leq \ceil{a}$
and so~$\sqrt{a} p \sqrt{a} \in \ceil{a}\mathscr{A}\ceil{a}$.
Clearly~$pap \in p\mathscr{A}p$.
Hence
\begin{align*}
\ceil{\sqrt{a} p \sqrt{a}} (\ceil{a} \mathscr{A} \ceil{a})\ceil{\sqrt{a} p \sqrt{a}}
& = \ceil{\sqrt{a} p \sqrt{a}} \mathscr{A} \ceil{\sqrt{a} p \sqrt{a}} \\
\ceil{pap} (p\mathscr{A} p)\ceil{pap} & =
\ceil{pap} \mathscr{A} \ceil{pap}.
\end{align*}

\item
We have~$u r(u)u = r(uu^*)u = uu^*u = \ceil{\sqrt{a}p\sqrt{a}}u$.
\end{enumerate}
Now we see the diagram makes sense and commutes:
\begin{align*}
h_p(c_a(\ceil{a}x\ceil{a}))
    & = p \sqrt{a} \ceil{a} x \ceil{a} \sqrt{a} p\\
    & = \sqrt{pap} u^* \ceil{a} x \ceil{a} u \sqrt{pap} \\
    & = \sqrt{pap} u^* \ceil{\sqrt{a}p\sqrt{a}}\ceil{a} x
            \ceil{a}\ceil{\sqrt{a}p\sqrt{a}} u \sqrt{pap} \\
    & = c_{pap}(\Ad_u (h_{\ceil{\sqrt{a}p\sqrt{a}}} (\ceil{a}x\ceil{a}))).
\end{align*}
Consequently~$h_p \after c_a$ is pure.
\end{proof}
\begin{rem}
    On the category of von Neumann algebra with pure maps,
        one may define a dagger which turns it into a dagger category.
    It is not directly clear it is unique,
        but with additional assumptions is can be shown to be unique.
    This is beyond the scope of this paper and will appear elsewhere.
\end{rem}

\begin{prop}\label{prop:paschkecompression}
    Let~$\varphi\colon \mathscr{A} \to \mathscr{B}'$
        be an NCP-map
        with Paschke dilation
            $\xymatrix{\mathscr{A}
            \ar[r]|-{\varrho} & \mathscr{P}
            \ar[r]|-{f} & {\mathscr{B}'}}$.
    Let~$b \in [0,1]_{\mathscr{B}}$
        together with a compression~$c\colon \mathscr{B}' \to \mathscr{B}$
            for~$b$.
    Then~$\xymatrix{\mathscr{A}
            \ar[r]|-{\varrho} & \mathscr{P}
            \ar[r]|-{c\after f} & \mathscr{B}}$
        is a Paschke dilation of~$c \after \varphi$.
\end{prop}
\begin{proof}
    Assume~$\mathscr{P}'$ is any von Neumann algebra
        together with NMIU-map~$\varrho' \colon \mathscr{A} \to \mathscr{P}'$
        and~NCP-map~$f'\colon \mathscr{P}' \to \mathscr{B}$
            such that~$f' \after \varrho' = c \after \varphi$.
    Note~$f'(1) = f'(\varrho'(1)) = c (\varphi(1)) \leq c(1) \leq b$.
    Hence there is a unique
    NCP-map~$f''\colon \mathscr{P}'\to \mathscr{B}'$
        with~$c \after f'' = f'$.
    Observe~$c \after f'' \after \varrho' = f' \after \varrho'
            = c \after \varphi$
        and so~$f'' \after \varrho' = \varphi$
            as~$c$ is injective.
    There is a unique~$\sigma \colon \mathscr{P}' \to \mathscr{P}$
        with~$\sigma \after \varrho' = \varrho$
            and~$f \after \sigma = f''$.
        But then~$c \after f \after \sigma
                        = c \after f'' = f'$
            and so we have shown existence of a mediating map.
    To show uniqueness, assume~$\sigma'\colon \mathscr{P}' \to \mathscr{P}$
        is any NCP-map such that~$c \after f \after \sigma' = f'$
            and~$\sigma\after \varrho' = \varrho$.
        Clearly~$c \after f \after \sigma' = f' = c \after f''$
                and so~$f \after \sigma' = f''$.
            Thus~$\sigma=\sigma'$ by definition of~$\sigma$.
\end{proof}

\begin{cor}\label{cor:fispure}
    Let~$\varphi \colon \mathscr{A} \to \mathscr{B}$
        be an NCP-map
        with Paschke dilation
            $\xymatrix{\mathscr{A}
            \ar[r]|-{\varrho} & \mathscr{P}
            \ar[r]|-{f} & \mathscr{B}}$.
    Then~$f$ is pure.
        Furthermore, if~$\varphi$ is contractive, then so is~$f$
            and if~$\varphi$ is unital, then~$f$ is a corner.
\end{cor}
\begin{proof}
    First, we will prove that if~$\varphi$ is unital,
        then~$f$ is a corner.
        Thus, assume~$\varphi$ is unital.
    Let $\xymatrix{\mathscr{A}
            \ar[r]|-{\varrho_s} & \mathscr{P}_s
            \ar[r]|-{f_s} & \mathscr{B}}$
    be the standard Paschke dilation of~$\varphi$.
    By \cite[Corollary 5.3]{bew154}, $f_s$ is a corner.
    By Lemma~\ref{lem:dilationsareisomorphic},
        we have~$f = f_s \after \vartheta$
        for some isomorphism~$\vartheta$.
    But then~$f$ is also a corner.

    Now, we will prove that for arbitrary contractive~$\varphi$,
        the map~$f$ is pure and contractive.
    The non-contractive case follows by scaling.
    Write~$\varphi'\colon \mathscr{A} \to
            \ceil{\varphi(1)}\mathscr{B}\ceil{\varphi(1)}$
        for the unique unital
    NCP-map such that~$\varphi = c_{\varphi(1)}\after \varphi'$.
    Let $\xymatrix{\mathscr{A}
            \ar[r]|-{\varrho'} & \mathscr{P}'
            \ar[r]|-{f'} & {\ceil{\varphi(1)}\mathscr{B}\ceil{\varphi(1)}}}$
        denote the Paschke dilation of~$\varphi'$.
    By Proposition~\ref{prop:paschkecompression},
        $\xymatrix{\mathscr{A}
            \ar[r]|-{\varrho'} & \mathscr{P}'
            \ar[rr]|-{c_{\varphi(1)}\after f'} && \mathscr{B}}$
    is a Paschke dilation of~$\varphi$.
        Thus by Lemma~\ref{lem:dilationsareisomorphic},
            we know~$f= c_{\varphi(1)} \after f' \after \vartheta$
                for some isomorphism~$\vartheta$.
        As these are all pure, $f$ is pure as well.
\end{proof}

\begin{thm}\label{thm:paschkecorner}
    Let~$\mathscr{A}$ be a von Neumann algebra together with
        a projection~$p \in \mathscr{A}$.
    Then a Paschke dilation of the standard corner~$h_p$
    is given by
    $\xymatrix{\mathscr{A}
    \ar[r]|-{h_{C_p}} & C_p\mathscr{A}
    \ar[r]|-{h_p} & p\mathscr{A}p}$.
\end{thm}
\begin{proof}
    Let $\xymatrix{\mathscr{A}
\ar[r]|-{\varrho} & A(\mathscr{A} \otimes_{h_p} p\mathscr{A}p)
\ar[r]|-{f} & p\mathscr{A}p}$ 
be the Paschke dilation of~$\varphi$
from Theorem~\ref{thm:paschke}.
The plan is to first prove that~$\mathscr{A}\otimes_{h_p} p\mathscr{A}p$
can be identified with~$\mathscr{A}p$,
and then to show that~$\A(\mathscr{A}p)=C_p\mathscr{A}$.

Note that~$\mathscr{A}p$ is a right $p\mathscr{A}p$-module
by~$(ap)\cdot(pbp) = apbp$,
and a pre-Hilbert $p\mathscr{A}p$-module
via~$\left< ap, \alpha p\right> = pa^* \alpha p$.
Since by the $C^*$-identity
the norm on~$\mathscr{A}p$
as a pre-Hilbert $\mathscr{B}$-module
coincides with the norm of~$\mathscr{A}p$
as subset of the von $C^*$-algebra~$\mathscr{A}$,
and~$\mathscr{A}p$ is norm closed in~$\mathscr{A}$,
and~$\mathscr{A}$ is norm complete,
we see that~$\mathscr{A}p$ is complete,
and thus a Hilbert $p\mathscr{A}p$-module.

The next step is to show that~$\mathscr{A}p$ is self-dual.
Let~$\tau\colon \mathscr{A}p \to p\mathscr{A}p$
be a bounded $p\mathscr{A}p$-module map.
We must find~$\alpha p \in \mathscr{A}p$
with~$\tau(ap) = p \alpha^* ap$
for all~$a\in\mathscr{A}$.
This requires some effort.
\begin{enumerate}
\item
We claim that~$C_p = \sup\{r\colon \,r \lesssim p\}$,
where~$r\lesssim p$
denotes that~$r$ is
a projection
which is
\emph{von Neumann-Murray below}~$p$,
i.e.~$r=vv^*$ and $v^*v\leq p$
for some~$v\in\mathscr{A}$,
see~\cite[Def.~6.2.1]{kr}.
To begin, 
writing~$q=\sup\{r\colon \,r\lesssim p\}$,
we have~$C_q=C_p$.
Indeed,
since~$p\lesssim p$, we have $p\leq q$, and so~$C_p\leq C_q$.
For the other direction, $C_q\leq C_p$,
note that if~$r$ is a projection with~$r\lesssim p$,
then~$r\leq C_r\leq C_p$ by~\cite[Prop.~6.2.8]{kr}.
Thus~$q\leq C_p$,
and so~$C_q\leq C_p$.

Thus we must prove that~$C_q-q=0$.
It suffices to show that $C_{C_q-q}=0$.
Note that for every projection~$r$ 
with $r\leq C_q-q$
and~$r\lesssim p$
we have~$r=0$,
because $r\lesssim p$ implies $r\leq q$ and so~$2r\leq C_q$.
Thus,
by~\cite[Prop.~6.1.8]{kr},
we get 
$C_{C_q-q} C_p = 0$.
But since~$C_q-q \leq C_q = C_p$,
we have $C_{C_q-q} \leq C_p$,
and so~$C_{C_q-q} = C_{C_q-q} C_p =0$.

    \item
Using Zorn's lemma,
we can find a 
family $(q_i)_{i\in I}$
of pairwise orthogonal projections
in~$\mathscr{A}$
with $q_i \lesssim p$
and~$C_p\equiv q=\sum_{i\in I}q_i$.
(Here,
and in the remainder of this proof, infinite sums in von Neumann
algebras are taken with respect to the ultraweak topology.)
For each~$i\in I$,
pick~$v_i\in\mathscr{A}$
with~$v_iv_i^* = q_i$
and~$v_i^*v_i \leq p$.
Since $p^\perp v_i^*v_i p^\perp \leq p^\perp p p^\perp=0$,
we have~$ v_ip^\perp =0$ by the $C^*$-identity,
and so~$v_i \in \mathscr{A}p$
for all~$i\in I$.

\item
Our plan is to prove that
$\tau(ap) = \left<(\sum_{i\in I}\tau(v_i)v_i^*)^*,ap\right>$
for all~$a\in\mathscr{A}$,
but first we must show that $\sum_{i\in I} \tau(v_i)v_i^*$
converges ultraweakly.
This requires a slight detour.
Let~$J\subseteq I$ be any finite subset.
Note that~$(\tau(v_i)^* \tau(v_j))_{ij}$
    is a matrix over~$p\mathscr{A}p$
    where~$i$ and~$j$ range over~$J$.
By~\cite[Thm.~2.8 (ii)]{bew154},
we have for all~$(b_i)_{i \in J}$
from~$p\mathscr{A}p$,
\begin{equation*}
\sum_{i,j\in J}b_i^* \tau(v_i)^*\tau(v_j) b_j \  =\ 
\tau(\sum_{i\in J} v_ib_i)^*
\tau(\sum_{i\in J} v_ib_i) \ \leq\ 
\| \tau \| \bigl<
    \sum_{i \in J} v_ib_i,
    \sum_{i \in J} v_ib_i
\bigr>
\ =\ 
\sum_{i,j\in J} b_i^*\left<v_i, v_j\right> b_j,
\end{equation*}
and so~$(\tau(v_i)^*\tau(v_j))_{ij}
    \leq \| \tau \| (\left<v_i, v_j\right>)_{ij}$
    as matrices over~$p\mathscr{A} p$
by~\cite[Prop.~6.1]{bew154}.
Since the projections $(q_i)_{i\in I}$
are pairwise orthogonal,
and $\left<v_i, v_j\right> = v_i^*v_j = v_i^*q_iq_jv_j$,
    we see that
$ (\left<v_i, v_j\right>)_{ij}$ is a diagonal matrix
below $p$, and since 
$\left<v_i,v_i\right>=v_i^*v_i \leq p$ 
we get
\begin{equation*}
\sum_{i,j \in J} v_i \tau(v_i)^* \tau(v_j) v^*_j
                \ \leq \ \|\tau\|  \sum_{i\in J } v_i p v_i^*
            \ =\  \|\tau\| \sum_{i \in J} q_i \ \leq\  \|\tau \|1.
\end{equation*}
From this it follows
that the net of partial sums of~$\tau(v_i)v_i^*$
is
norm bounded (by the $C^*$-identity),
and ultraweakly Cauchy
(by Cauchy--Schwarz and the fact that~$\sum_{j\in J} q_j$
converges ultraweakly as~$J$ increases),
and thus ultraweakly convergent\cite[Prop.~40]{ww16}.
Define~$\alpha \equiv  (\sum_i \tau(v_i)v_i^*)^*$.

\item
\label{item:inpap}
Note that $p\alpha^*= \sum_i p \tau(v_i)v_i^* = \alpha^*$
(because $\tau(v_i)\in p\mathscr{A}p$)
and so~$\alpha \in \mathscr{A}p$.

\item
The linear map~$\tau((\,\cdot\,) p)\colon \mathscr{A} \to \mathscr{A}$
is ultrastrongly continuous,
because if~$a_ip \to 0$ ultrastrongly,
then for any normal state~$\omega$ on~$p\mathscr{A}p$
we have
\begin{equation*}
\omega(\tau(a_ip)^*\tau(a_ip)) 
 \ \leq\  \| \tau \| \omega(\left< a_ip, a_ip \right>)
        \ =\  \|\tau\| \omega( pa_i^* a_i p )\ \rightarrow\  0.
\end{equation*}

\item
Pick any~$ap\in \mathscr{A}p$.
As~$q = \sup_i q_i$,
we have~$q = \sum_i q_i$ ultrastrongly
(combine \cite[Lemma~5.1.4]{kr} with \cite[Prop.~1.15.2]{sakai}),
and thus $\tau(qap) = \sum_i \tau (q_i ap)$
ultraweakly.
Hence
\begin{equation*}
    \tau( ap)
    \ =\  \tau( q ap)
    \ =\   \sum_i \tau(q_i ap)
    \ =\   \sum_i \tau(v_ipv_i^* ap)
    \ =\  \sum_i \tau(v_i)pv_i^* ap
    \ =\  \left< \alpha p,ap\right>,
\end{equation*}
where we use that~$qap=ap$ (since $q=C_p$),
and $q_i = q_i^2 = v_iv_i^*v_iv_i^*=v_ipv_i^*$.
Thus~$\mathscr{A}p$ is self-dual.
\end{enumerate}

Now we will show~$\mathscr{A}p$ is isomorphic
to~$\mathscr{A}\otimes_{h_p} p\mathscr{A}p$.
Note~$[ap\otimes p,ap \otimes p] = pa^*ap$
and so~$ap \otimes p \in N$
if and only if~$ap = 0$.
A straight-forward computation
shows~$a \otimes p\alpha p - a p \alpha p \otimes p \in N$
and so every~$x \in \mathscr{A} \odot p \mathscr{A}p$
is $N$-equivalent to exactly one~$ap \otimes p$ for some~$a \in \mathscr{A}$.
Thus~$ap\otimes p + N \mapsto ap$
fixes an isomorphism~$X_0 \to \mathscr{A}p$.
As~$\mathscr{A}p$ is already complete and self-dual,
so is~$X_0$.
Hence via~$ap \mapsto \widehat{ap\otimes p}$
we have~$\mathscr{A}p \cong X_0 \cong X \cong X' \equiv
        \mathscr{A} \otimes_{h_p} p\mathscr{A}p$
and so~$\xymatrix{\mathscr{A}
\ar[r]|-{\varrho} & \A(\mathscr{A}p)
\ar[r]|-{f} & p\mathscr{A}p}$
is a Paschke dilation for~$h_p$,
where~$\varrho(\alpha) ap = \alpha ap$
and~$f(t)= pt(p)$.

If in the proof above ---
that~$\tau \equiv \alpha^*(\,\cdot\,)p$
for some~$\alpha\in \mathscr{A}p$ --- 
we replace~$\tau$
by a 
bounded $p\mathscr{A}p$-module
map $t\colon \mathscr{A}p\to\mathscr{A}p$,
then the reasoning is still valid
(except for point~\eqref{item:inpap}),
and so we see that
the map~$\varrho
\colon\mathscr{A}\to \A(\mathscr{A}p)$
given by~$\varrho(\alpha_0)(ap)=\alpha_0ap$
is surjective.

Now we show~$\car \varrho = C_p$.
It is sufficient to show
that for each~$\alpha\geq 0$ with~$\alpha\in \mathscr{A}$,
we have that~$\alpha ap = 0$  for all~$a \in \mathscr{A}$
if and only if~$\alpha C_p=0$.
The reverse direction follows from~$\alpha a p = \alpha (1-C_p) ap = 0$
whenever~$\alpha C_p = 0$.
Thus assume~$\alpha ap = 0$ for all~$a \in \mathscr{A}$.
In particular~$0 = \alpha v_i p v_i^* = \alpha q_i$
and so~$\alpha C_p = \alpha q = \sum_i \alpha q_i = 0$,
as desired.
We now know~$\varrho = \vartheta \after h_{C_p}$
for some isomorphism~$\vartheta \colon \A(\mathscr{A}p) \to C_p\mathscr{A}$.
It is easy to see~$f\after \vartheta^{-1} = h_p$
and so we have proven our Theorem.
\end{proof}

\begin{cor}\label{cor:purequiv}
    Let~$\varphi \colon \mathscr{A} \to \mathscr{B}$
        be an NCP-map
        with Paschke dilation
$\xymatrix{\mathscr{A}
\ar[r]|-{\varrho} & \mathscr{P}
\ar[r]|-{f} & \mathscr{B}}$.
    Then the map $\varphi$ is pure if and only if~$\varrho$ is surjective.
\end{cor}
\begin{proof}[Proof of Corollary~\ref{cor:purequiv}]
Assume~$\varrho$ is surjective.
The kernel of~$\varrho$
is~$z \mathscr{A}$ for some central projection~$z$.
(See e.g.~\cite[1.10.5]{sakai}.)
Note the quotient-map for the kernel of~$\varrho$
is a corner for~$z$,
hence by the isomorphism theorem~$\varrho$ is a corner.
By Corollary~\ref{cor:fispure} $f$ is pure.
Thus~$\varphi$ is the composition of pure maps and hence pure.

Now assume~$\varphi$ is pure.
By scaling, we may assume~$\varphi$ is contractive.
Write~$p \equiv \car \varphi$.
Note $\varphi = c \after h_p$
for some compression~$c\colon p\mathscr{A}p \to \mathscr{B}$
and standard corner~$h_p$ for~$p$.
$\xymatrix{\mathscr{A}
\ar[r]|-{h_{C_p}} & C_p\mathscr{A}
\ar[r]|-{h_p} & \mathscr{pAp}}$
is a Paschke dilation for~$h_p$
and so by Proposition~\ref{prop:paschkecompression},
$\xymatrix{\mathscr{A}
\ar[r]|-{h_{C_p}} & \mathscr{zA}
\ar[r]|-{c\after h_p} & \mathscr{B}}$
is a Paschke dilation for~$\varphi$.
By Lemma~\ref{lem:dilationsareisomorphic},
we know~$\varrho = \vartheta \after h_{C_p}$
for some isomorphism~$\vartheta$
and so~$\varrho$ is surjective.
\end{proof}

Now we have a better grip on when a Paschke embedding is surjective.
The following is a characterization of when a Paschke embedding
is injective.  This is a generalization of our answer\cite{stineinj} to the
same question for the Stinespring embedding.
\begin{thm}
    Let~$\varphi \colon \mathscr{A} \to \mathscr{B}$
        be an NCP-map
        with Paschke dilation
$\xymatrix{\mathscr{A}
\ar[r]|-{\varrho} & \mathscr{P}
\ar[r]|-{f} & \mathscr{B}}$.
    The map $\varrho$ is injective if and only
        $\varphi$ maps no non-zero central projection to zero.
            (Equivalently: $C_{\car \varphi} = 1$.)
\end{thm}
\begin{proof}
    Let~$\alpha \in \mathscr{A}$.
    As a first step,
    we claim~$\varrho(\alpha) = 0$
        if and only~$\varphi(a^* \alpha^*\alpha a)=0$
            for every~$a \in \mathscr{A}$.
        From left to right is easy (expand~$\varrho(\alpha) a \otimes 1$).
    To show the converse,
            assume~$\varphi(a^* \alpha^* \alpha a) = 0$
                for all~$a \in \mathscr{A}$.
    Then for every sequence~$a_1, \ldots, a_n \in \mathscr{A}$
        we may use~$a := \sum_i a_i$
        and so~$\sum_{i,j}\varphi(a_j^* \alpha^*\alpha a_i) = 0$.
    From this and \cite[Prop.~6.1]{bew154} 
        it follows that for 
            every sequence~$b_1, \ldots, b_n \in \mathscr{B}$,
        we have~$\sum_{i,j} b_j^* \varphi(a_j^* \alpha^* \alpha a_i)b_i=0$.
    Hence~$\varrho(\alpha) \sum_i a_i \otimes b_i = 0$
            for any~$\sum_i a_i \otimes b_i$.
    This is sufficient to conclude~$\varrho(\alpha) = 0$, as desired.

    Assume~$\varrho$ is injective.
    For brevity, write~$p := \car \phi$.
    Let~$a \in \mathscr{A}$ be given.
    Note
        \begin{equation*}
a^* (1-C_p) a = (1-C_p)a^*a(1-C_p) \leq \| a\|^2 (1-C_p)
                \leq \|a\|^2 (1-p)
        \end{equation*}
    and so~$\varphi(a^* (1-C_p) a) \leq \|a\|^2\varphi(1-p)= 0$.
    By the initial claim, we see $\varrho(1-C_p)=0$ and
            so~$C_p=1$, as desired.

    For the converse, assume~$C_p=1$ and~$\varrho(\alpha)=0$.
    By Zorn's lemma, find a maximal family of orthogonal
        projections~$(q_i)_{i \in I}$ from~$\mathscr{A}$
        with~$q_i \lesssim p$.
    Then~$\sup_{i\in I} q_i = C_p$; see point 2 from
    the proof of Thm.~\ref{thm:paschkecorner}.
        For each~$q_i$ pick a~$v_i$ such
            that~$v_i v_i^* = q_i$ and~$v_i^*v_i \leq p$.
    From~$\varrho(\alpha)=0$ we saw it follows that
        $\varphi(a^* \alpha^* \alpha a)=0$ for all~$a \in \mathscr{A}$.
    In particular~$\varphi(v_i^* \alpha^* \alpha v_i)=0$.
    Without loss of generality we may assume~$a^*a \leq 1$
        and then~$v_i^* \alpha^*\alpha v_i \leq 1-p$.
        Hence
\begin{equation*}
q_i \alpha^*\alpha q_i = v_iv_i^* \alpha^*\alpha v_i v_i^*
                \leq v_i (1-p) v_i^* = v_i v_i^* - v_iv_i^*=0.
\end{equation*}
    Consequently~$\alpha^*\alpha \leq 1-q_i$ for every~$i \in I$.
        Thus~$\alpha^* \alpha \leq 1- \sup_{i\in I} q_i = 1-C_p=0$.
    By the C$^*$-identity,~$\alpha=0$ and so~$\varrho$ is indeed injective.
\end{proof}

To continue our study of pure maps,
we need some preparation.
\begin{dfn}
    Let~$\varphi\colon \mathscr{A} \to \mathscr{B}$ be any NCP-map.
\begin{enumerate}
    \item
Write $[0,\varphi]_\NCP \equiv \{\psi\colon \mathscr{A} \to \mathscr{B}
                \text{ NCP-map};\ 
        \varphi - \psi \text{ is completely positive} \}.$
    \item
    We say~$\varphi$ is \keyword{NCP-extreme},
        if it is an extreme point among the NCP-maps
            with same value on~$1$;
    that is: $\lambda \varphi_1 + (1-\lambda)\varphi_2 = \varphi$
            for~$\varphi_1,\varphi_2\colon \mathscr{A} \to \mathscr{B}$
                NCP-maps and~$0<\lambda<1$
        implies~$\varphi_1=\varphi_2=\varphi$
    \item
If $\xymatrix{\mathscr{A}
\ar[r]|-{\varrho} & \mathscr{P}
\ar[r]|-{f} & \mathscr{B}}$  is a Paschke dilation of~$\varphi$
and~$t \in \varrho(\mathscr{A})'$ with  $t \geq 0$,
define~$\varphi_t\colon \mathscr{A} \to \mathscr{B}$
by~$\varphi_t(a) \equiv f(\sqrt{t} \varrho(a) \sqrt{t})$.
\end{enumerate}
\end{dfn}

\begin{thm}\label{thm:correspondence}
Assume $\varphi\colon \mathscr{A} \to \mathscr{B}$ is an NCP-map
        with Paschke dilation
$\xymatrix{\mathscr{A}
\ar[r]|-{\varrho} & \mathscr{P}
\ar[r]|-{f} & \mathscr{B}}$.
\begin{enumerate}
    \item
        The map $t \mapsto \varphi_t$
            is an affine order isomorphism~$[0,1]_{\varrho(\mathscr{A})'}
                        \to [0,\varphi]_\NCP$.
    \item $\varphi$ is NCP-extreme
            if and only if~$t \mapsto \varphi_t(1)$
            is injective on~$[0,1]_{\varrho(\mathscr{A})'}$.
\end{enumerate}
\end{thm}
\begin{proof}
The Paschke dilation constructed
in Theorem~\ref{thm:paschke}
is called the standard Paschke dilation,
for which
point 1 is shown in~\cite[Prop.~5.4]{bew154}
and point 2 in~\cite[Thm.~5.4]{bew154}.
We show the result carries to an arbitrary Paschke dilation.
Write
$\xymatrix{\mathscr{A}
\ar[r]|-{\varrho_s} & \mathscr{P}_s
\ar[r]|-{f_s} & \mathscr{B}}$
for the standard Paschke dilation of~$\varphi$.
Write~$\varphi_t^\mathscr{P}$
and~$\varphi_t^{\mathscr{P}_s}$
to distinguish between~$\varphi_t$
relative to the given and standard Paschke dilation.
Let~$\vartheta\colon \mathscr{P} \to \mathscr{P}_s$
be the mediating isomorphism
from Lemma~\ref{lem:dilationsareisomorphic}.

It is easy to see~$\vartheta$
restricts to an affine order
isomorphism~$[0,1]_{\varrho(\mathscr{A})'} \to [0,1]_{\varrho_s(\mathscr{A})'}$.
Note
\begin{equation*}
\varphi^{\mathscr{P}}_t (a) = f(\sqrt{t} \varrho(a) \sqrt{t})
                = f_s (\vartheta(\sqrt{t} \varrho(a) \sqrt{t}))
                = f_s (\sqrt{\vartheta(t)} \varrho_s(a) \sqrt{\vartheta(t)})
                = \varphi^{\mathscr{P}_s}_{\vartheta(t)}(a)
\end{equation*}
and so~$t \mapsto \varphi^{\mathscr{P}}_t$
is the composition of affine order
isomorphisms~$\vartheta^{-1}$
and~$t \mapsto \varphi^{\mathscr{P}_s}_t$,
which proves 1.
Finally, for 2,
note~$t \mapsto \varphi_t^{\mathscr{P}_s}(1)$
is injective if and only
if~$t \mapsto \varphi_t^{\mathscr{P}}(1)$
is as~$\vartheta^{-1}$ is  injective.
\end{proof}

\begin{prop}\label{prop:stormerpure}
    Let~$\varphi \colon \mathscr{A} \to B(\mathscr{H})$
        be any NCP-map.
    Then~$\varphi$
        is pure in the definition of St\o rmer \cite[Def.~3.5.4]{stormer}
        if and only if~$\varphi$ is pure
        as in Def.~\ref{dfn:pure}.
\end{prop}
\begin{proof}
    Let~$\varphi\colon \mathscr{A} \to B(\mathscr{H})$ be any NCP-map
        with standard Paschke dilation~$\xymatrix{\mathscr{A}\ar[r]|\varrho &
            \mathscr{P}
    \ar[r]|f & B(\mathscr{H})}$.
        Note that by Theorem~\ref{thm:stinespring-paschke},
        we know~$\mathscr{P}\cong B(\mathscr{K})$ for some
            Hilbert space~$\mathscr{K}$
        and so it is a factor.

    Assume~$\varphi\colon \mathscr{A} \to B(\mathscr{H})$
        is pure as in Def.~\ref{dfn:pure}.
    Let~$\psi \colon \mathscr{A} \to B(\mathscr{H})$
        with~$\varphi - \psi$ completely positive.
    To show~$\varphi$ is pure in the sense of St\o rmer,
        we have to show~$\psi = \lambda \varphi$
        for some~$\lambda \in [0,1]$.
    By \cite[Prop.~5.4]{bew154},
        $\psi = \varphi_t$
        for some~$t \in \varrho(\mathscr{A})'$
        with~$0 \leq t \leq 1$.
    In particular~$\psi$ is normal.
    As~$\varphi$ is pure
        we know by Corollary~\ref{cor:purequiv}
        that~$\varrho$ is surjective.
    Thus~$\varrho(\mathscr{A})' = Z(\mathscr{P}) = \C 1$.
    Thus~$t = \lambda 1$ for some~$\lambda \in [0,1]$.
    We conclude~$\psi = \varphi_t = \varphi_{\lambda1}
                = \lambda \varphi_1 = \lambda \varphi$
        as desired.

    For the converse, assume~$\varphi$ is St\o rmer pure.
        By Corollary~\ref{cor:purequiv},
        it is sufficient to show~$\varrho$ is surjective.
        As~$\varrho(\mathscr{A})$
            is a von Neumann subalgebra of~$\mathscr{P}$,
            we may conclude that~$\varrho$ is surjective
            if we can show~$\varrho(\mathscr{A})' \subseteq Z(\mathscr{P})$
            as~$\mathscr{P} = Z(\mathscr{P})'
                        \subseteq \varrho(\mathscr{A})''
                        = \varrho(\mathscr{A}) \subseteq \mathscr{P}$
            by the double commutant theorem.
    To this end, let~$t \in \varrho(\mathscr{A})'$.
    Without loss of generality, we may assume~$0 \leq t \leq 1$.
        Then~$\varphi_t \in [0,1]_\NCP$
            and so~$\varphi_t = \lambda \varphi$ for some~$\lambda \in [0,1]$.
        Hence~$\varphi_t = \lambda \varphi = \varphi_{\lambda 1}$
            and so~$t = \lambda 1 \in Z(\mathscr{P})$,
        which completes the proof.
\end{proof}

\begin{thm}
NMIU-maps and pure NCP-maps are NCP-extreme.
\end{thm}
\begin{proof}
Let~$\varrho \colon \mathscr{A} \to \mathscr{B}$
be any NMIU-map.
$\xymatrix{\mathscr{A}
\ar[r]|-{\varrho} & \mathscr{B}
\ar[r]|-{\id} & \mathscr{B}}$ is a Paschke dilation of~$\varrho$.
If~$t \in \varrho(\mathscr{A})'$,
then~$\varrho_t(1) = t$
and so, clearly, $t \mapsto \varrho_t(1)$
is injective on~$\varrho(\mathscr{A})'$.
Hence by Theorem~\ref{thm:correspondence},
we see~$\varrho$ is NCP-extreme.

Assume~$\varphi\colon \mathscr{A} \to \mathscr{B}$ is pure.
By scaling, we may assume~$\varphi$ is contractive.
Write~$p \equiv \car \varphi$.
Then~$\varphi = c \after h_p$ for some
compression~$p \mathscr{A} p \to \mathscr{B}$.
By Proposition~\ref{prop:paschkecompression}
and Theorem~\ref{thm:paschkecorner}
we know
$\xymatrix@C+.4pc{\mathscr{A}
\ar[r]|-{h_{C_p}} & C_p\mathscr{A}
\ar[r]|-{c \after h_p} & \mathscr{B}}$ is a Paschke dilation of~$\varphi$.
We have to show~$t \mapsto \varphi_t(1)$ is injective
on~$[0,1]_{h_{C_p}(\mathscr{A})'}$.
As~$h_{C_p}$ is surjective,
 we have~$h_{C_p}(\mathscr{A})' = Z(C_p \mathscr{A})$
hence~$\varphi_t(1) = c(h_p(\sqrt{t} h_{C_p}(1)\sqrt{t}))
                    = c(pt p )$ for~$t \in Z(C_p \mathscr{A})$.
As compressions are injective,
it is sufficient to show~$h_p$
is injective on~$Z(C_p \mathscr{A})$.
Assume~$t \in Z(C_p \mathscr{A})$
such that~$ptp=0$.
Then~$0=r(pt)=pr(t)$ and so~$p \leq 1-r(t)$.
As~$1-r(t)$ is a central projection,
we must have~$1-r(t) \leq C_p = 1_{C_p \mathscr{A}}$.
Hence~$r(t)=0$ and so~$t=0$.  We are done.
\end{proof}

\begin{prob}
Is there an NCP-extreme NCP-map into a factor
which is neither pure
nor a compression after an NMIU-map?
\end{prob}

\subsection*{Remarks}
The construction of the standard Paschke dilation
is a generalization of the GNS-construction
to Hilbert C$^*$-modules.
For this reason, for those studying C$^*$-modules,
the construction is known as Paschke's GNS
(e.g.~\cite[Remark~8.4]{skeide1}).
There is also a generalization of Stinespring
to~C$^*$-modules due to Kasparov.\cite{kasparov}
This Theorem, however, is like Stinespring only applicable
to NCP-maps of which the codomain has a certain form
and hence as far as it applies to arbitrary NCP-maps,
it reduces to Paschke.

\subsubsection*{Acknowledgments}
We thank Robin Adams, Aleks Kissinger, Hans Maassen,
Mathys Rennela, Michael Skeide, and Sean Tull
for their helpful suggestions.
We especially thank
Chris Heunen for receiving the first author on a research visit,
for suggesting Proposition~\ref{prop:stinespring-spatial}
(which was the starting point of this paper),
for Example~\ref{exa:pure},
and for many other contributions which will appear
in a future publication.
We have received funding from the
European Research Council under grant agreement \textnumero~320571.

\bibliography{main}{}
\bibliographystyle{eptcs}

\end{document}